\documentclass[12pt]{article}
\usepackage{graphicx} %
\usepackage{caption}
\usepackage{subcaption}
\usepackage{url}
\usepackage{amsmath}
\usepackage{amsthm}
\usepackage{amsfonts}
\usepackage{amssymb}
\usepackage{listings}
\usepackage{geometry}
\usepackage{tikz}
\usepackage{bm}

\usepackage{todonotes}
\usepackage{algorithm}
\usepackage{algpseudocode}
\usepackage{enumitem}
\usepackage{natbib}
\usepackage{xcolor}
\newtheorem{lemma}{Lemma}
\newtheorem{Theorem}{Theorem}
\newtheorem*{Remark}{Remark}

\newtheorem{definition}{Definition}

\newtheorem{Proposition}{Proposition}
\title{Robust Optimal Portfolio in a Mixture Setting with Partial Ambiguity}
\author{N. D. Shyamalkumar and Tianrun Wang\thanks{Corresponding author: tianrun-wang@uiowa.edu}}
\date{\today}

\def\xr{\mathbf{x}'\mathbf{R}}
\def\x{\mathbf{x}}
\def\CVaR#1{{\rm CVaR}_p\left(#1\right)}

\def\R{\mathbf{R}}
\def\r{\mathbf{r}}
\DeclareMathOperator*{\argmin}{argmin}
\DeclareMathOperator*{\argmax}{argmax}

\def\E#1{\mathbb{E}{\left(#1\right)}}
\def\EP#1#2{\mathbb{E}_{#1}{\left(#2\right)}}
\def\Var#1{{\rm Var}\left(#1\right)}
\def\VarP#1#2{{\rm Var}_{#1}\left(#2\right)}

\makeatletter
\DeclareRobustCommand{\cev}[1]{%
  {\mathpalette\do@cev{#1}}%
}
\newcommand{\do@cev}[2]{%
  \vbox{\offinterlineskip
    \sbox\z@{$\m@th#1 x$}%
    \ialign{##\cr
      \hidewidth\reflectbox{$\m@th#1\vec{}\mkern4mu$}\hidewidth\cr
      \noalign{\kern-\ht\z@}
      $\m@th#1#2$\cr
    }%
  }%
}
\makeatother

\usepackage{stackengine}
\stackMath
\setstackgap{L}{.4ex}

\def\calS{{\mathcal{S}}}

\def\calN{{\mathcal{N}}}
\def\calP{{\mathcal{P}}}

\def\Real{\mathbb{R}}

\def\E#1{\mathbb{E}{\left(#1\right)}}
\def\EP#1#2{\mathbb{E}_{#1}{\left(#2\right)}}
\def\Var#1{{\rm Var}\left(#1\right)}

\usepackage{marginnote}
\usepackage{etoolbox}

\newtoggle{lmargin}
\newcommand{\alternatingtodo}[2][]{%
    \iftoggle{lmargin}%
    {%
        \todo[#1]{#2}%
        \togglefalse{lmargin}%
    }{%
        {%
            \let\marginpar\marginnote%
            \reversemarginpar%
            \todo[#1]{#2}%
        }%
        \toggletrue{lmargin}%
    }%
    \ignorespaces%
}

\usepackage[normalem]{ulem}

\begin{document}
\maketitle

\begin{abstract}

Managing insurance and financial risk when data is limited is a key task in the insurance industry. In this paper, we focus on cases where the risk distribution is modeled as a mixture with some components estimable to high precision or known, and others, along with their weights, are not. Our paper addresses two robust portfolio optimization problems with partial ambiguity, where the loss function involves either variance or conditional value-at-risk (CVaR). We use a projected subgradient descent algorithm to solve the optimization problems.  The problem reduces to a convex-nonconcave minimax problem. We show that, while the general problem converges at an $O(1/\sqrt{k})$ rate, where $k$ denotes the number of iterations, exponential convergence is possible in some cases. Lastly, we provide numerical examples to show the effectiveness of our approach and the attainment of a geometric convergence rate. This work aims to provide more effective solutions for actuarial decision-making under model uncertainty.

\end{abstract}

\section{Introduction}

Since the mid-2000s, US insurance companies have been subject to risk-based capital regulations under the C3 Phase II framework (see \citet{AAA2005C3PhaseII}) that require stochastic simulation to calculate regulatory capital for products such as variable annuities. Currently, US insurance companies conduct extensive stochastic simulations every quarter to calculate statutory reserves and capital for multiple products under VM-20, VM-21, and VM-22, see \citet{NAIC2024ValuationManual}. Such use of stochastic simulation to evaluate risk measures is inherently passive, as it is not intended to drive strategic business decisions. Alternatively, let us consider an internal asset-liability management exercise to choose among a few business strategies that could significantly alter the liability portfolio in the next year. The goal is to decide on a strategy that maximizes a risk-adjusted profit measure. Note that each such strategy will also involve adaptive choice of its optimal asset portfolio. This ALM task differs from the above RBC calculation in two important ways. First, the computation of the risk-adjusted profit involves a stochastic programming exercise in determining the optimal asset portfolio. Second, while the stochastic programming exercise will involve using a finite set of samples akin to the RBC task, the quality of the optimal asset portfolio will crucially depend on how representative the samples are of the future economic environment. 

We now state a mathematical formulation of the above problem. Assuming we know the true distribution $P_*$ of a random vector $\mathbf{X}$, for a function $h$, we seek to solve 
\[
\inf_{\theta \in \Theta}  \mathbb{E}_{P_{*}} \left[ h(\theta, \mathbf{X}) \right],
\]
where $\theta$ is the decision variable, and $\Theta$ is the set of possible values $\theta$ can take. However, the true distribution is often unknown, and we only have past data, or a set of scenarios in actuarial regulatory parlance, say $\mathbf{x}_1,\ldots,\mathbf{x}_n$. A natural solution is to use the {\it empirical} distribution of the past data or the discrete uniform distribution supported on the {\it scenarios} as an approximant of the true distribution. In other words, we get the Sample-Average Approximation (SAA) problem
\[
\inf_{\theta \in \Theta}  \frac{1}{n}\sum_{i=1}^n h(\theta, \mathbf{x}_i).
\]
Unfortunately, the optimal solutions to the SAA problem often suffer from overfitting, also known as the optimizer's curse (see \citet{kuhn2019wasserstein}). Specifically, in the context of our motivating problem, such a determined optimal asset allocation may have relied on idiosyncratic features of the approximant distribution for its performance. And, in the absence of such features in the true distribution, it may perform much worse. 

A way to alleviate this issue is to modify the objective function to account for performance on distributions in a neighborhood of the approximant, in the hope that dependence on idiosyncrasies for superior performance can be weakened, resulting in {\it robust} optimal allocations. Such a solution to this problem is the min-max Distributionally Robust Optimization (DRO) problem
\[
\inf_{\theta \in \Theta} \sup_{Q \in \calN(\hat{P}_{\mathbf{X}})} \mathbb{E}_Q \left[ h(\theta, \mathbf{X}) \right],
\]
where $\hat{P}_{\mathbf{X}}$ is the discrete uniform distribution on $\mathbf{x}_1,\ldots,\mathbf{x}_n$, and $\calN(\hat{P}_{\mathbf{X}})$ denotes a neighborhood of distributions around it, called the ambiguity set. There are tradeoffs to consider in the choice of the ambiguity set. For example, it should be large enough to mitigate overfitting and small enough that one does not incur significant underperformance on the true distribution, especially as the cost of attaining robust performance for implausible distributions. 

There are many approaches to specifying the neighborhood $\calN(\hat{P}_{\mathbf{X}})$ of distributions around a best {\it guess} $\hat{P}_{\mathbf{X}}$ across many fields and applications.  One of the earliest in classical statistical robustness was the $\epsilon$-contamination model of \citet{Huber1964}, which was later carried over to Bayesian prior robustness by \citet{BergerBerliner1986}; also see \citet{Berger1994} for an overview of many different neighborhood specifications used in the context of Bayesian prior robustness. In actuarial science, there is literature on finding conservative risk bounds - robust risk measurement - when some part of the distribution is constrained or specified; for example, the marginals are constrained, and the dependence structure is arbitrary, or moments are constrained, and the distribution is an arbitrary unimodal distribution. See  \citet{embrechts2005worst}, \citet{BERNARD20209}, \citet{Li-Shao-Wang-Yang}, \citet{CORNILLY2018141}, and references therein. In the context of robust stochastic programming, see \citet{delage2010distributionally} for constraints on the first moment to lie in a given ellipsoid and the second moment to satisfy a positive semi-definite constraint. As mentioned in \citet{Qihe}, a shortcoming of such bounds is that reasonable constraints on the {\it functional} values - such as moments and measures of dependence -  used in defining these neighborhoods tend to allow implausible distributions resulting in overly conservative risk bounds.

An alternative method of specifying the neighborhood $\calN(\hat{P}_{\mathbf{X}})$ is to take it to be a ball in a chosen metric space topology on the space of distributions, or defined by the use of f-divergences (see \citet{csiszar1972class}). f-divergences, of which the most popular is the Kullback-Leibler divergence, have fundamental connections to statistical inference; see, for example, \citet{KullbackLeibler1951, Akaike1973} and \citet{Berk1966}. Also, in the setting of robust stochastic programming, an important property is that asymptotic confidence sets can be constructed as f-divergence neighborhoods of the empirical distribution, see \citet{BenTal2013}. Despite these advantages, the fact that f-divergences often blow up unless the distribution is absolutely continuous with respect to the reference distribution is unappealing, as centered on the empirical, it would exclude the true distribution. For this reason, while still maintaining mathematical tractability, we use Wasserstein distances in our development; see \citet{kuhn2019wasserstein} and \citet{blanchet2021statistical} for tutorials on DRO using Wasserstein distances. Other options include the kernel-based Maximum Mean Discrepancy (MMD) distance (see \citet{StaibJegelka2019}) and the class of Integral Probability Metrics (IPM) 
(see \citet{Muller1997} and \citet{Sriperumbudur2012}, which includes the total variation distance from the f-divergences, MMD, and the Wasserstein-1 distance; see \citet{Husain2020} for the use of IPM in the DRO setting. For $\mu$ a probability measure on $\Real^d$, the Wasserstein ball of order $\rho$ with radius $r$, centered at $\mu$ is defined as $\mathbb{B}_{W_\rho}(\mu,r)=\{\nu:W_p(\mu,\nu)<r\}$, where $W_p(\mu,\nu)^p=\inf_{\mathbf{X}\sim \mu,\mathbf{Y}\sim \nu}\E{||\mathbf{X}-\mathbf{Y}||_p^p}$ with the infimum taken over all couplings of $\mu$ and $\nu$.

We combine the $\epsilon$-contamination approach of \citet{Huber1964} with the contaminant distribution restricted to a Wasserstein ball to construct the ambiguity set in our DRO setting. Of course, in the setting of \citet{Huber1964}, the contaminant distribution is that of the outliers and is naturally uninformative. Instead, following \citet{Qihe}, in our setting the unknown distribution of $\mathbf{X}$ is assumed to have a mixture structure, with some components known and others unknown, and the component weights are also unknown. \citet{Qihe} considers an upper bound for $\mathbb{E}_{P_*} \left[ h(X) \right]$ by 
\begin{equation}
\label{problem}
    \sup_{Q \in \calN(\hat{F}_X)} \mathbb{E}_Q \left[ h(X) \right],
\end{equation} 
for some polynomial function $h$ and an ambiguity set $\calN$ as described above.

We make several contributions to the literature. We apply this mixture ambiguity framework to robust Mean-Variance and Mean-CVaR portfolio optimization problems, provide algorithms to solve them, and analyze their convergence rates.
Specifically, denote the decision variable to be the investment portfolio $\mathbf{x} \in \Real^d$, and the return to be a random vector $\R \in \Real^d$. We want to select an investment portfolio $\mathbf{x}$, so as to minimize the worst-case risk:

\begin{equation} \label{investment}
    \inf_{\mathbf{x}:\mathbf{1}'\mathbf{x}=1}\sup_{F_\R \in \calN( \hat{F}_\R)} D(-\mathbf{x}'\R),
\end{equation}
where $D$ is a disutility function and $\hat{F}_\R$ has a two-mixture structure, coming from possibly two economic regimes. For example, the return $\R$ could be modeled as a mixture distribution, with the weight unknown and the return in the high-interest-rate regime subject to ambiguity due to limited data. We employ an ambiguity set $\calN$ that respects this mixture structure, which we discuss in detail in Sections 2 and 3.

In Section 2, we study \eqref{investment} with $D(\cdot)=\Var{\cdot}+\gamma \E{\cdot}$, the Mean-Variance DRO problem, derive its equivalent min-max formulation, and solve it using a projected subgradient descent algorithm. The algorithm is guaranteed to converge at an $O(\frac{1}{\sqrt{T}})$ rate, and, under additional conditions, we show that it can attain an exponential convergence rate as well. In Section 3, we extend our framework to the Mean-CVaR problem with $D(\cdot)=\E{\cdot}+\rho\CVaR{\cdot}$. In Section 4, we present numerical examples demonstrating the attainment of exponential rates even when some assumptions are violated. Additionally, we demonstrate the efficacy of our DRO method by comparing out-of-sample disutility with that of the SAA method. In both cases, we show, both theoretically and empirically, that as the radius of our ball tends to infinity, the DRO portfolio converges to the 1/N portfolio. In Section 5, we conclude and discuss the case when a short sale is permitted.

\textbf{Notation:}  Let $\rightrightarrows$ denote uniform convergence. If $f$ is a function, $f_x:=\frac{\partial}{\partial x}f$ is the partial derivative of $f$ with respect to $x$. $\|\cdot\|$, $\|\cdot\|_1$, and $\|\cdot\|_\infty$ denote the 2-norm, 1-norm, and $\infty$-norm of a vector, respectively. For a matrix $A$, $\|A\|_2:=\max_{\x\neq\mathbf{0}} \frac{\|A\x\|}{\|\x\|}$ denotes the operator 2-norm of $A$. 

\section{Mixture for Mean-Var Disutility function}

Consider the problem of constructing an investment portfolio from $d$ assets. We represent the proportion of wealth invested in the assets by a vector of weights $\mathbf{x} \in \Delta$,  where $$\Delta=\{\x \in \Real^d:\x \ge 0, \mathbf{1}'\x =1\},$$
is the $d$-simplex. The portfolio return is given by $\mathbf{x}'\mathbf{R}$, where $\mathbf{R}$ is the vector of asset returns. Our objective is to choose $\mathbf{x}$ to minimize a disutility function, with the true distribution of $\R$ subject to a two-component mixture structure, with a component and the weight unknown. Note that, although we do not allow short positions in the development below, we discuss relaxing this restriction in Section 5. %

In this section, we begin by considering the following mean-variance disutility function, which reflects a trade-off between the expected return and the risk of the portfolio, with risk measured by its variance and the trade-off controlled by the parameter $\gamma> 0$:
\begin{equation}
\Var{\x'\R} - \gamma\E{\x'\R}.
\end{equation}
Note that $\frac{2}{\gamma}$ is the traditional coefficient of risk aversion in mean-variance utility, {\it e.g.} see Chapter 2 of \citet{ang2014asset}. The unknown true distribution of the asset returns, $P$, is assumed to belong to an ambiguity set of probability measures denoted by $\calS$. The portfolio selection problem is thus the following worst-case disutility minimization problem:

\begin{equation}
\tag{Opt 1}
\inf_{\x \in \Delta} \sup_{P \in \calS} \left( \VarP{P}{\mathbf{x}'\mathbf{R}} - \gamma\EP{P}{\mathbf{x}'\mathbf{R}} \right).
\label{eq:main_problem}
\end{equation}

An ambiguity set $\calS$, a set of probability measures, is a representation of the ambiguity in modeling.  For a market with two potential regimes, with a regime and the relative occurrence of the two regimes subject to ambiguity, we consider the following form of ambiguity sets: 
\begin{equation}
\label{ambiset}
\calS = \{ P = (1-q)P_N + qP_S \mid q \in [q_0-\varepsilon,q_0+\varepsilon], \, P_S \in \mathbb{B}_{W_2}(P_{S_0}, r(q)) \},
\end{equation}
where  $P_N,P_{S_0} \in \calP(\Real^d)$, the set of Borel probability measures on $\Real^d$, and $\varepsilon>0$. For our development below, we further assume that $r \in C^2([0,1])$, and that the variance-covariance matrices $\Sigma_{P_N}$ and  $\Sigma_{P_{S_0}}$, under the two regimes, are positive definite. In other words, we assume that the assets are constrained to form a minimal set, free of multicollinearity, under both regimes. We note that, to include the risk-free asset, it is equivalent to replacing the probability simplex $\Delta$ by the sub-probability simplex. Since compactness is retained by this substitution, our development carries through to this case.

The definition of $\calS$ is determined by judgment, guided by domain knowledge and the out-of-sample environment,  on which distributions are plausible, as well as the mathematical tractability of the resulting computational problem. In this manner, robustness against an unknown sampling environment can be achieved at an acceptable cost of conservatism, see \citet{blanchet2025distributionally}. The above defined $\calS$ can be embedded in a Wasserstein ball, since $r$ is bounded. However, such a Wasserstein ball would contain distributions far from the above structural form, especially when $P_N$ and $P_{S_0}$ are quite far apart in the $W_2$ distance. As a result, a Wasserstein ball will introduce an unacceptable level of conservatism in achieving the desired robustness.

In the above definition, %
$\mathbb{B}_{W_2}(P_{S_0}, r(q))$ is the Wasserstein-2 ball of radius $r(q)$ around a reference distribution $P_{S_0}$. $P_{S}$, where $S$ stands for "Stress scenario", accounts for atypical/stressed investment environments, such as those involving high interest rates, the environment surrounding a market crash, and other scenarios in which market fundamentals have strayed significantly from the norm. Clearly, from the very description of such a regime, a modeler has insufficient data to model $P_{S}$, especially relative to that available for modeling $P_N$, where $N$ stands for "normal". The relative occurrence of a stressed market environment is denoted by $q_0$. For example, a simple back-of-the-envelope calculation could be as follows. Over the last 40 years, there have been three significant market crashes: the dot-com crash of 2002, the subprime mortgage crisis of 2008, and the COVID-19 pandemic of 2020, suggesting the use of $q_0=(3/40)=0.075$.

 Theorem \ref{minimax} below reduces \eqref{eq:main_problem} to a minimax problem, the proof of which is in the next subsection. Towards that, we define two auxiliary functions that will be used throughout the section:
 \begin{equation}
     \psi_{a,\gamma}(y) := (y-a)^2 - \gamma y,
 \end{equation}
 and
 \begin{equation}
     V(q, \mathbf{x},a)=\left( r(q)\|\mathbf{x}\| + \sqrt{\x'\Sigma_{P_{S_0}} \x + \left(\x' \mu_{P_{S_0}} - \frac{2a+\gamma}{2}\right)^2} \right)^2 - a\gamma - \frac{\gamma^2}{4}.
 \end{equation}
 
 \newpage
 
\begin{Theorem}
\label{minimax}
    The optimization problem in \eqref{eq:main_problem} is equivalent to the following min-max problem:
\begin{equation}
\label{finalopt}
\min_{\x \in \Delta,a \in I}\;\max_{q \in [q_0-\varepsilon,\, q_0+\varepsilon]} \bigg[ (1-q)\cdot\EP{P_{N}}{\psi_{a,\gamma}(\mathbf{x}'\mathbf{R})} + q\cdot V(q, \mathbf{x},a) \bigg],
\end{equation}
where $I$ is a compact interval not dependent on $\x$.
\qed
\end{Theorem}

We propose a projected subgradient descent algorithm (Algorithm \ref{algo1}) to solve \eqref{finalopt}. Define the min-max objective
\begin{equation}\label{hdef}
h(q,\x,a):=(1-q)\cdot \EP{P_{N}}{\psi_{a,\gamma}(\mathbf{x}'\mathbf{R})} + q\cdot V(q, \mathbf{x},a),
\end{equation}
and its partial maximum 
\begin{equation}\label{Jdef}
J(\x,a) := \sup_{q \in [q_0-\varepsilon,q_0+\varepsilon]} h(q,\x,a), 
\end{equation}
with $\Omega:=[q_0-\varepsilon,\, q_0+\varepsilon]\times I \times \Delta$, the compact domain of $h$. We assume that $\varepsilon<q_0 \wedge (1-q_0)$, bounding $q$ away from both $0$ and $1$.
The projection step for this algorithm projects a vector onto $\Delta$, with the projection operator conveniently expressed as the solution of the following quadratic program:
\begin{equation}
    \label{proj}
    \text{Proj}_\Delta(\mathbf{y}) = \argmin_{\mathbf{x} \in \Delta} \frac{1}{2} \|\mathbf{x} - \mathbf{y}\|^2, \quad \mathbf{y}\in\Real^d.
\end{equation}

\begin{algorithm}
\caption{Projected Subgradient Descent for Robust Mean-Var \label{algo1}}
\begin{algorithmic}[1]
\State \textbf{Initialize:} Choose $(\mathbf{x}_0,a_0) \in \Real\times \Delta$, the number of iterations $T>0$ and step sizes $\eta_k > 0$ for $k=0,1,\ldots,T$.
\For{$k=0, 1, \dots, T$}
    \State Find the worst-case probability ${q}^*_k = \argmax\limits_{q\ni |q-q_0|\leq \epsilon} h(q,\x_k,  a_k)$.
    \State Compute  the gradient $\mathbf{g}_{k, \mathbf{x}},g_{k, a}$ of $h$ with respect to $(\x,a)$ at $({q}^*_k,\x_k,a_k)$. 
    \State Update the variables:
    \begin{align*}
        \mathbf{x}_{k+} &\leftarrow \mathbf{x}_k - \eta_k \mathbf{g}_{k, \mathbf{x}}\\
            a_{k+1} &\leftarrow \tau_k - \eta_k g_{k, a}
    \end{align*}
    \State Project $\mathbf{x}_{k+1} \leftarrow \text{Proj}_{\Delta}(\mathbf{x}_{k+})$.
\EndFor
\State \textbf{Return} $(a_1,\x_1),\ldots,(a_T,\mathbf{x}_T)$.
\end{algorithmic}
\end{algorithm}

Among projected descent algorithms, the projection step is usually the main computational bottleneck; however, in our case, we have a closed-form expression for the projection operator. For $\mathbf{y}\in\Real^d$, $\mathbf{x}^*:=\text{Proj}_\Delta(\mathbf{y})$ is given by  
\begin{equation}
x_i^* = \max(y_i - \lambda, 0), \quad \hbox{where } \lambda = \frac{1}{m}\left(\sum_{j=1}^m y_{(j)} - 1\right),
\end{equation}
with $m$ being the largest integer such that 
\[
y_{(m)} - \frac{1}{m}\left(\sum_{j=1}^m y_{(j)} - 1\right) > 0.
\]
In the above,  $y_{(1)} \ge y_{(2)} \ge \dots \ge y_{(d)}$ denote the ordered components of the vector $\mathbf{y}$, and 
can be computed in $O(d \log d)$ time. This is also the complexity of the projection step.

Note that in step 4 of Algorithm \ref{algo1}, what we really seek is a subgradient $\mathbf{g}_k$ of $J$ at  $(\x_k,a_k)$. In view of Lemma \ref{propertieshJ}, by Danskin's Theorem (see Theorem \ref{danskin} in the Appendix), such a subgradient $\mathbf{g}_k$ is given by the gradient of the inner function $h$ with respect to $(\x,a)$, evaluated at $(q_k^*,\x_k,a_k)$. The following proposition, the proof of which is in the Appendix, gives an expression for this subgradient. For its statement, we define  $S_\gamma$ as,  
\[S_\gamma(\x,a) = \sqrt{\x'\Sigma_{P_{S_0}}\x + \left(\x'\mu_{P_{S_0}} - \frac{2a+\gamma}{2}\right)^2},
\]
and note that $h$ can be expressed using it as   
\[
    h(q,\x,a)=(1-q)\EP{P_{N}}{\psi_{a,\gamma}(\mathbf{x}'\mathbf{R})} 
    + q \left[\left( r(q)\|\mathbf{x}\| + S_\gamma(\x,a) \right)^2 - a\gamma - \frac{\gamma^2}{4}\right].
\]

\begin{Proposition}
\label{subgradientJ}
    A subgradient of $J$ at $(\x_k,a_k)$ is $(\mathbf{g}_{k, \mathbf{x}},g_{k, a})$, where
    \begin{multline*}
    \mathbf{g}_{k, \mathbf{x}} = (1-q_k^*)\left[2(\Sigma_{P_N} + \mu_{P_N}\mu_{P_N}')\mathbf{x}_k - 2a_k\mu_{P_N} - \gamma\mu_{P_N}\right] \\
    + q_k^* \left[ 2\left( r(q_k^*)\|\mathbf{x}_k\| + S_{\gamma}(\x_k,a_k) \right) \left(r(q_k^*) \frac{\mathbf{x}_k}{\|\mathbf{x}_k\|}+\frac{\Sigma_{P_{S_0}} \x_k+\left(\x_k' \mu_{P_{S_0}} - \frac{2a_k+\gamma}{2}\right) \cdot \mu_{P_{S_0}}}{S_{\gamma}(\x_k,a_k)}\right) \right],
\end{multline*}
and 
\begin{equation*}
    g_{k, a} = (1-q_k^*)[-2(\mathbf{x}_k'\mu_{P_N} - a_k)] + q_k^*\left[-2\frac{r(q_k^*)\|\mathbf{x}_k\| + S_{\gamma}(\x_k,a_k)}{S_{\gamma}(\x_k,a_k)}\left(\x_k' \mu_{P_{S_0}} - \frac{2a_k+\gamma}{2}\right) - \gamma \right].
\end{equation*}
\qed
\end{Proposition}

Note that since $r \in C^2([0,1])$, we have $h \in C^{2}(\Omega)$. Thus $h_q, h_{qq}, \frac{\partial^2}{\partial q \partial (\x,a)} h$ and $\frac{\partial^2}{\partial ^2 (\x,a)} h$ are continuous functions on the compact domain $\Omega$. 
Let the maximum 2-norm of $(\mathbf{g}_{k, \mathbf{x}}, g_{k, a})$ over $(q_k^*, \x_k,a_k) \in \Omega$ be $K_1$. Let $K_2, K_3$ be positive constants 
\[
K_2 := \max_{(q,\x,a) \in \Omega}\left\lVert\frac{\partial^2}{\partial (\x,a)^2} h \right\rVert_2, \quad K_3:=\max_{(q,\x,a) \in \Omega} \left\|\frac{\partial^2}{\partial q \partial (\x,a)} h\right\|.
\] 
Note that for $K_1, K_2, K_3$, it suffices to use an upper bound instead of the maximum.

In Lemma \ref{propertieshJ} below, we establish convexity and smoothness properties of $h$ and 
$J$. 
Note that if $f$ is twice differentiable, then $\alpha$ strong convexity (see Definition \ref{sconv} in the Appendix) is equivalent to the eigenvalues of its Hessian being at least $\alpha$. On the other hand, $\beta$-smoothness (see Definition \ref{smooth} in the Appendix) is equivalent to these eigenvalues being at most $\beta$. The ratio $\beta/\alpha$ is known as the condition number of $f$. If $-f$ is strongly convex, then we say $f$ is strongly concave.  The proof of Lemma \ref{propertieshJ} is in the next subsection.

\newpage

\begin{lemma} We have the following statements regarding functions $h$ and $J$:
    \label{propertieshJ}
    \begin{enumerate}[label={\roman*.}]
        \item   $h$  is strongly convex in $(\x,a)$. 
        \item   $J$ is $K_1$-Lipschitz and $((1-q_0-\varepsilon)\lambda_M)$-strongly convex, where $\lambda_M$ is the smallest eigenvalue of \[M:=\begin{pmatrix}\Sigma_{P_{N}}+\mu_{P_{N}}\mu_{P_{N}}' & -\mu_{P_{N}}\\[4pt]-\mu_{P_{N}}' & 1\end{pmatrix}.
        \]
        \item    When for any $(\x,a)$, $h$ is strongly concave in $q$ on $ [q_0-\varepsilon,q_0+\varepsilon]$, then there exists $\alpha>0$ such that $h_{qq} \leq -\alpha<0 $, and 
        $J$ is $\left(\frac{K_3^2}{\alpha}+K_2\right)$-smooth. %
    \end{enumerate}
\qed
\end{lemma}

Three fundamental complexity results for the projected subgradient descent algorithm are  Theorem \ref{slowConverge}, \ref{mediumConverge}, and Theorem \ref{fastConverge} of the Appendix, which lay out conditions under which 
the algorithm converges at a $O(1/\sqrt{k})$, $O(1/k)$ and an exponential rate, respectively.  
From Lemma \ref{propertieshJ} ii., $J$ is strongly convex, and hence there exists a unique point, denoted below by $(\x^*,a^*)$, at which it attains its minimum. Also, since a subgradient of $J$, as in Proposition \ref{subgradientJ}, is bounded by $K_1$, the conditions for Theorem \ref{mediumConverge} are satisfied by $J$. Hence, our algorithm has a convergence rate of at least  $O(1/k)$. Moreover, if conditions of Lemma \ref{propertieshJ} iii. hold, $\nabla J$ will be Lipschitz, and then from Theorem \ref{fastConverge} we can conclude that the convergence rate will be exponential. We formalize the discussion of convergence rate in the following two theorems.

\begin{Theorem}
For the ambiguity sets, as defined in \eqref{ambiset}, Algorithm \ref{algo1} with time-varying step size
\[
\eta_k = \frac{2}{\lambda_M(1-q_0-\varepsilon) (k+1)}
\]
satisfies
\[
J\left( \sum_{k=1}^T \frac{2k}{T(T+1)} (\x_k,a_k) \right) - J((\x^*,a^*))
\;\le\;
\frac{2K_1^2}{\lambda_M(1-q_0-\varepsilon) (T+1)} .
\]
\end{Theorem}
\begin{proof}
    This is a combination of Lemma \ref{propertieshJ} ii. and Theorem \ref{mediumConverge}.
\end{proof}

In Theorem \ref{fastConvergeVar} below, we show that Algorithm \ref{algo1} can attain an exponential rate of convergence. Our formulation of this theorem is based on choosing $r$, for convenience,  to peak at $q_0$. Such a function can be chosen from the versatile beta distribution family
\[
r(q) \propto q^{\alpha-1}(1-q)^{\beta-1},
\]
with $r$ attaining its maximum at $\frac{\alpha-1}{\alpha+\beta-2}$ when $\alpha>1,\beta>1$, and unimodality implying locally strictly concave near its mode. The proof of Theorem \ref{fastConvergeVar} is in the next subsection.

\begin{Theorem}
\label{fastConvergeVar}
 If $r'(q_0)=0$, $r''(q_0)<0$, then there exists an $\varepsilon_0>0$, such that for those ambiguity sets defined in \eqref{ambiset} with $\varepsilon \leq \varepsilon_0$, Algorithm \ref{algo1} with a constant step size $$\eta_k=\frac{K_4}{K_3^2+K_2K_4}$$ satisfies
 \[
J(\x_{T},a_T) - J(\x^*,a^*)  \leq \frac{K_3^2+K_2K_4}{2K_4}\, \exp\!\left(-\frac{K_4\,\lambda_M(1-q_0-\varepsilon_0)}{K_3^2+K_2K_4}\,T\right) \|(\x_{0},a_0) - (\x^*,a^*)\|^2,
\]
where $K_4:=-q_0 r(q_0) r''(q_0)/d>0$.
\qed
\end{Theorem}

Finally, Theorem \ref{oneOverD}, the proof of which is in the next subsection, shows that as the ambiguity becomes large, $\x^*$ approaches $(\frac{1}{d},\ldots,\frac{1}{d})$, the equal-weighted portfolio, also known as the $1/N$ portfolio, see \citet{guo2019does}.

\begin{Theorem}
\label{oneOverD}
Let $(\x_C^*, a_C^*)$ be a minimizer of $J$ with the ambiguity radius in \eqref{ambiset} defined by $r(q) = C \cdot r_{\text{base}}(q)$, where $C>0$ is a scaling parameter and $r_{\text{base}}$ is a fixed, non-negative continuous function on the interval $[0,1]$,  that is not identically zero on $[q_0-\varepsilon, q_0+\varepsilon]$. Then,  as $C \to \infty$, $\x_C^*$ converges to the equal weighted portfolio $\frac{1}{d} \cdot \mathbf{1}$.
\end{Theorem}

\subsection{Proofs of Results}
\subsubsection{Proof of Theorem \ref{minimax}}
To prove Theorem \ref{minimax}, we need the following lemmas.
We begin by deriving an alternate form for the inner maximization in \eqref{eq:main_problem}.

\begin{lemma}
\label{alternateFormOfOpt1}
For portfolio weights $\x$, the inner problem of \eqref{eq:main_problem} satisfies, 
\begin{equation}
    \sup_{P \in \calS} \left( \VarP{P}{\mathbf{x}'\mathbf{R}} - \gamma\EP{P}{\mathbf{x}'\mathbf{R}} \right) =
\min_{a \in I} \sup_{|q-q_0| \le \varepsilon} \left\{ (1-q)\EP{P_{N}}{\psi_{a,\gamma}(\mathbf{x}'\mathbf{R})} + q V(q, \x, a) \right\},
\end{equation}
where
\[
\psi_{a,\gamma}(y) := (y-a)^2 - \gamma y, \quad V(q, \x, a) := \sup_{G \in \mathbb{B}_{W_2}(P_{S_0}, r(q))} \EP{P_{S}}{\psi_{a,\gamma}(\mathbf{x}'\mathbf{R})},
\]
and $I$ is a compact interval not dependent on $\x$.
\end{lemma}

\begin{proof}
We use the variational form of variance, {\it i.e.} $\Var{Y} = \min_{a \in \mathbb{R}} \E{Y-a}^2$,  
\begin{equation}
    \sup_{P \in \calS} \left( \VarP{P}{\mathbf{x}'\mathbf{R}} - \gamma\EP{P}{\mathbf{x}'\mathbf{R}} \right) = \sup_{P \in \calS} \min_{a \in \Real} \left( \EP{P}{(\mathbf{x}'\mathbf{R} - a)^2 - \gamma \cdot \mathbf{x}'\mathbf{R}} \right).
\end{equation}

 We first show that the inner optimization can be restricted to a compact interval $I$ free of $P\in \calS$ and $\mathbf{x}\in\Delta$.
It suffices to show an $I$ containing the minimum point $\EP{P}{\x'\R}$ for all $P\in \calS$ and $\mathbf{x}\in\Delta$.  As f is continuous, let $M$ be the maximum of $f$ on $[q_0-\varepsilon,q_0+\varepsilon]$. Then for any $P_{S}$ in $\mathbb{B}_{W_2}(P_{S_0}, M)$ and $\mathbf{x}\in\Delta$, if we denote $\pi$ to be the optimal 2-Wasserstein coupling between $P_S$ and $P_{S_0}$ and $(\R,\R_0) \sim \pi$, we have
\begin{multline*}
    \EP{P_{S}}{\|\R\|^2} = \EP{\pi}{\|\R_0+\R-\R_0\|^2} \leq 2\left(\EP{\pi}{\|\R_0\|^2}+\EP{\pi}{\|\R-\R_0\|^2}\right)\\
    \leq 2\left(\EP{P_{S_0}}{\|\R\|^2}+M^2\right),
\end{multline*}

and 
\[
\left[\EP{P_{S}}{\x'\R}\right]^2 \leq\EP{P_{S}}{\x'\R}^2 \leq \|\x\|^2\EP{P_{S}}{\|\R\|^2} \leq 2(\EP{P_{S_0}}{\|\R\|^2}+M^2). 
\]
Similarly, 
$$
\left(\EP{P_{N}}{\x'\R}\right)^2 \leq\left(\EP{P_{N}}{(\x'\R)^2}\right) \leq \|\x\|^2\EP{P_{N}}{\|\R\|^2} \leq \EP{P_{N}}{\|\R\|^2}.
$$ 
Since,  any $P$ in $\calS$ is a mixture of $P_{N}$ and a $P_{S}$ in $\mathbb{B}_{W_2}(P_{S_0}, M)$, with mixture weight $q$ in $[q_0-\varepsilon,q_0+\varepsilon]$, combining the above statements, we have,
$$
\left\vert\EP{P}{\x'\R}\right\vert = \left\vert(1-q)\EP{P_{N}}{\x'\R}+q\EP{P_{S}}{\x'\R}\right\vert \leq \sqrt{2(\EP{P_{S_0}}{\|\R\|^2}+M^2)\vee 
 \EP{P_{N}}{\|\R\|^2}}.
$$
The above shows the existence of our desired $I$. 

Since $I$ is compact, $\EP{P}{(\mathbf{x}'\mathbf{R} - a)^2 - \gamma \cdot \mathbf{x}'\mathbf{R}} $ is convex and continuous in $a$, and concave (linear) and continuous in $P$ with respect to the 2-Wasserstein metric topology (see Theorem 7.12 in \citet{villani2003topics} for implications of $W_2$ convergence), from Sion's Minimax Theorem (see Theorem \ref{sion} in the Appendix), we can swap the min and sup operators:

\begin{equation}
\sup_{P \in \calS} \min_{a \in I} \left( \EP{P}{(\mathbf{x}'\mathbf{R} - a)^2 - \gamma \cdot \mathbf{x}'\mathbf{R}} \right)=\min_{a \in I} \sup_{P \in \calS} \left( \EP{P}{(\mathbf{x}'\mathbf{R} - a)^2 - \gamma \cdot \mathbf{x}'\mathbf{R}} \right).
\label{eq:swapped_problem}
\end{equation}
It follows that, for fixed $\x,a $, 
\[
\sup_{P \in \calS} \left( \EP{P}{(\mathbf{x}'\mathbf{R} - a)^2 - \gamma(\mathbf{x}'\mathbf{R})} \right)=\sup_{|q-q_0| \le \varepsilon} \left\{ (1-q)\EP{P_{N}}{\psi_{a,\gamma}(\mathbf{x}'\mathbf{R})} + q V(q, \x, a) \right\}.
\]
\end{proof}

The above lemma implies that for fixed $\x,a$, the inner problem can be solved sequentially; first by finding the worst-case distribution $G_q$ in the Wasserstein ball, and then by finding the worst-case mixture weight $q$.

We use the duality result of Theorem~1 in \citet{gao2023distributionally} (see Theorem \ref{duality} in the Appendix) to derive the following representation of $V(\cdot,\cdot,\cdot)$:
\begin{equation}
\label{V_dual}
V(q, \x, a) = \inf_{\lambda \ge 0} \left\{ \lambda r(q)^2 + \EP{P_{S_0}}{ \sup_{\r} \left( \psi_{a,\gamma}(\mathbf{x}'\r) - \lambda \|\r - \R\|^2 \right) } \right\},
\end{equation}
where $\R$ is a random vector drawn from the reference distribution $P_{S_0}$. The key to the tractability of the infimum in \eqref{V_dual} is a closed-form expression for the inner supremum. Towards this, we establish the following two lemmas.

\begin{lemma}
\label{inner r}
    \begin{multline}
    \label{EG0}
        \EP{P_{S_0}}{ \sup_{\r} \left( \psi_{a,\gamma}(\mathbf{x}'\r) - \lambda \|\r - \R\|^2 \right) }\\
        =\begin{cases}
             \frac{\lambda}{\lambda - \|\mathbf{x}\|^2} \left( \x'\Sigma_{P_{S_0}} \x + \left(\x' \mu_{P_{S_0}} - \frac{2a+\gamma}{2}\right)^2 \right) - a\gamma- \frac{\gamma^2}{4}, &\lambda \geq \|\mathbf{x}\|^2\\
             +\infty, &\lambda < \|\mathbf{x}\|^2,
        \end{cases}
    \end{multline}
    with $\mu_{P_{S_0}} ,\Sigma_{P_{S_0}}$ denoting the mean and variance-covariance matrix of $P_{S_0}$, and the understanding that $\frac{0}{0}:=0,\frac{c}{0}:=\infty \hbox{ for }c>0$.
\end{lemma}
\begin{proof}
We derive an expression for the supremum by first optimizing on $\mathbf{r}$ with the value of the portfolio return $\mathbf{x}'\r$ set to $y$, and then we optimize on $y$.
\begin{equation}
\sup_{\r} (\psi_{a,\gamma}(\mathbf{x}'\r) - \lambda \|\r - \R\|^2) = \sup_{y \in \mathbb{R}} \left\{ \psi_{a,\gamma}(y) - \inf_{\r:\mathbf{x}'\r=y} \lambda\|\r - \R\|^2 \right\}.
\end{equation}
The infimum is attained at 
$$\r^*=\R-\left(\frac{\x'\R-y}{\|\x\|^2} \right)\x,$$ making
\begin{equation}
\inf_{\r:\x'\r=y} \|\r - \R\| = \frac{|y - \mathbf{x}'\R|}{\|\mathbf{x}\|}.
\end{equation}
Substituting this back, we obtain a one-dimensional maximization problem over $y$. Denoting $\mathbf{x}'\R$ by $Y$, we now derive a closed form expression for, 
\begin{multline}
    H(\lambda,Y):=\sup_{y \in \mathbb{R}} \left\{ (y-a)^2 - \gamma y - \frac{\lambda}{\|\mathbf{x}\|^2}(y - Y)^2 \right\} \\
    =\sup_y  \left\{\left(1 - \frac{\lambda}{\|\mathbf{x}\|^2}\right)y^2 + \left(2\frac{\lambda Y}{\|\mathbf{x}\|^2} - 2a - \gamma\right)y + \left(a^2 - \frac{\lambda Y^2}{\|\mathbf{x}\|^2}\right)\right\}.
\end{multline}
The above quadratic maximization is concave if $\lambda > \|\mathbf{x}\|^2$. Simple calculations show
\begin{equation}
    H(\lambda,Y) = \begin{cases}
        \frac{\lambda}{\lambda - \|\mathbf{x}\|^2}\left(Y - \frac{2a+\gamma}{2}\right)^2 - a\gamma-\frac{\gamma^2}{4}, & \lambda \geq \|\mathbf{x}\|^2\\
        +\infty , & \lambda < \|\mathbf{x}\|^2
    \end{cases}
    \label{eq:H_y0_simplified}
\end{equation}
with the above-defined extended real number convention. Combining, we get \eqref{EG0} by using the fact that  
\[
\mathbb{E}_{P_{S_0}}[Y]=\x' {\mu}_{P_{S_0}}, \quad \hbox{and} \quad  
 \text{Var}_{P_{S_0}}(Y)=\x'\Sigma_{P_{S_0}} \x, 
\]
 in $\EP{P_{S_0}}{H(\lambda,Y)}$.
\end{proof}

\begin{lemma}
\label{v(q,x,a)}
\[
V(q, \x, a)=\left( r(q)\|\mathbf{x}\| + \sqrt{\x'\Sigma_{P_{S_0}} \x + \left(\x' \mu_{P_{S_0}} - \frac{2a+\gamma}{2}\right)^2} \right)^2 - a\gamma - \frac{\gamma^2}{4}
\]
\end{lemma}

\begin{proof}
From Lemma \ref{inner r},
    for a fixed $\x,a,q$, 
\begin{equation}
    V(q, \x, a) = \inf_{\lambda \ge \|\mathbf{x}\|^2} \left\{ \lambda r(q)^2 + \frac{\lambda}{\lambda - \|\mathbf{x}\|^2} \left( \x'\Sigma_{P_{S_0}} \x + \left(\x' \mu_{P_{S_0}} - \frac{2a+\gamma}{2}\right)^2 \right) - a\gamma- \frac{\gamma^2}{4} \right\}.
\end{equation}
Let $\alpha:=\x'\Sigma_{P_{S_0}} \x + \left(\x' \mu_{P_{S_0}} - \frac{2a+\gamma}{2}\right)^2$. The optimal $\lambda$ in minimizing the function
\begin{equation}
    L(\lambda) = \lambda r(q)^2 + \left(\frac{\lambda}{\lambda - \|\mathbf{x}\|^2}\right) \alpha , \quad \text{subject to } \lambda \geq \|\mathbf{x}\|^2.
\end{equation}
 is given by
\begin{equation}
    \lambda^*  = \|\mathbf{x}\|^2 + \frac{\|\mathbf{x}\|\sqrt{\alpha}}{r(q)},
\end{equation}
and the optimal value is given by
\begin{align*}
    \min_{\lambda \geq \|\x\|^2} L(\lambda) = (r(q)\|\mathbf{x}\| + \sqrt{\alpha})^2.
\end{align*}
Substituting the expression for $\alpha$ back gives the following closed-form solution:
\begin{equation}
    V(q, \x, a) = \left( r(q)\|\mathbf{x}\| + \sqrt{\x'\Sigma_{P_{S_0}} \x + \left(\x' \mu_{P_{S_0}} - \frac{2a+\gamma}{2}\right)^2} \right)^2 - a\gamma - \frac{\gamma^2}{4}.
\end{equation}
\end{proof}

{Combining Lemma \ref{alternateFormOfOpt1}, \ref{inner r}, \ref{v(q,x,a)}, Theorem \ref{minimax} is proved
}

\subsubsection{Proof of Proposition \ref{subgradientJ}}
A subgradient of $J$ at $(\x_k,a_k)$ is $\mathbf{g}_k = \nabla_{(\x,a)} h(q_k^*, \mathbf{x}_k, a_k)$. Recall our notation $$S_{\gamma}(\x_k,a_k) = \sqrt{\x'\Sigma_{P_{S_0}} \x + \left(\x' \mu_{P_{S_0}} - \frac{2a_k+\gamma}{2}\right)^2}.$$
\textbf{Calculating a subgradient with respect to $\mathbf{x}$:} 
$$\nabla_\mathbf{x} \EP{P_{N}}{\psi_{a,\gamma}(\mathbf{x}'\mathbf{R})}=2(\Sigma_{P_N} + \mu_{P_N}\mu_{P_N}')\mathbf{x}_k - 2a_k\mu_{P_N} - \gamma\mu_{P_N}.$$

\begin{align*}
    \nabla_\mathbf{x} V(q_k^*,  \mathbf{x}_k, a_k) &= 2\left( r(q_k^*)\|\mathbf{x}_k\| + S_{\gamma}(\x_k,a_k) \right) \cdot \nabla_\mathbf{x} \left( r(q_k^*)\|\mathbf{x}_k\| + S_{\gamma}(\x_k,a_k) \right) \\
    &= 2\left( r(q_k^*)\|\mathbf{x}_k\| + S_{\gamma}(\x_k,a_k) \right) \left(r(q_k^*) \frac{\mathbf{x}_k}{\|\mathbf{x}_k\|}+\nabla_x S_{\gamma}(\x_k,a_k)\right),
\end{align*}

where $$\nabla_x S_{\gamma}(\x_k,a_k)= \frac{\Sigma_{P_{S_0}} \x+\left(\x' \mu_{P_{S_0}} - \frac{2a_k+\gamma}{2}\right) \cdot \mu_{P_{S_0}}}{S_{\gamma}(\x_k,a_k)}$$
A subgradient of $J$ with respect to $\x$ is the weighted sum:
\begin{multline*}
    \mathbf{g}_{k, \mathbf{x}} = (1-q_k^*)\left[2(\Sigma_{P_N} + \mu_{P_N}\mu_{P_N}')\mathbf{x}_k - 2a_k\mu_{P_N} - \gamma\mu_{P_N}\right] + \\q_k^* \left[ 2\left( r(q_k^*)\|\mathbf{x}_k\| + S_{\gamma}(\x_k,a_k) \right) \left(r(q_k^*) \frac{\mathbf{x}_k}{\|\mathbf{x}_k\|}+\frac{\Sigma_{P_{S_0}} \x+\left(\x' \mu_{P_{S_0}} - \frac{2a_k+\gamma}{2}\right) \cdot \mu_{P_{S_0}}}{S_{\gamma}(\x_k,a_k)}\right) \right].
\end{multline*}
\textbf{Calculating a subgradient with respect to $a$:}
 $$\frac{\partial}{\partial a} \EP{P_{N}}{\psi_{a,\gamma}(\mathbf{x}'\mathbf{R})}=-2(\mathbf{x}_k'\mu_{P_N} - a_k).$$
 
\begin{equation*}
    \frac{\partial S_{\gamma}(\x_k,a_k)}{\partial a_k} = \frac{1}{2S_{\gamma}(\x_k,a_k)} \cdot 2\left(\x' \mu_{P_{S_0}} - \frac{2a_k+\gamma}{2}\right) \cdot \left(-\frac{2}{2}\right) = -\frac{1}{S_{\gamma}(\x_k,a_k)}\left(\x' \mu_{P_{S_0}} - \frac{2a_k+\gamma}{2}\right).
\end{equation*}
Using the chain rule on $V$:
\begin{align*}
   \frac{\partial}{\partial a_k} V(q_k^*,  \mathbf{x}_k, a_k) &= 2\left( r(q_k^*)\|\mathbf{x}_k\| + S_{\gamma}(\x_k,a_k) \right) \frac{\partial S_{\gamma}(\x_k,a_k)}{\partial a_k} - \gamma \\
    &= -2\,\frac{r(q_k^*)\|\mathbf{x}_k\| + S_{\gamma}(\x_k,a_k)}{S_{\gamma}(\x_k,a_k)}\left(\x' \mu_{P_{S_0}} - \frac{2a_k+k}{2}\right) - \gamma.
\end{align*}
A subgradient of $J$ with respect to $a$ is the weighted sum:
\begin{equation*}
    g_{k, a} = (1-q_k^*)[-2(\mathbf{x}_k'\mu_{P_N} - a_k)] + q_k^*\left[-2\,\frac{r(q_k^*)\|\mathbf{x}_k\| + S_{\gamma}(\x_k,a_k)}{S_{\gamma}(\x_k,a_k)}\left(\x' \mu_{P_{S_0}} - \frac{2a_k+\gamma}{2}\right) - \gamma\right].
\end{equation*}

\subsubsection{Proof of Lemma \ref{propertieshJ}}

\begin{proof}
\begin{enumerate}[label=\roman*.]
    \item  Towards showing $h$ is strongly convex in $(\x,a)$, we note that 
    \begin{align*}
        \nabla_{(\x,a)}^2\EP{P_N}{\psi_{a,\gamma}(\mathbf{x}'\mathbf{R})}&=\nabla_{(\x,a)}^2\EP{P_N}{(\xr-a)^2-\gamma \xr}\\&=2\EP{P_N}{\begin{pmatrix}
    \R \\
    -1
    \end{pmatrix} (\R',-1)} \\
    &=\begin{pmatrix}\Sigma_{P_{N}}+\mu_{P_{N}}\mu_{P_{N}}' & -\mu_{P_{N}}\\[4pt]-\mu_{P_{N}}' & 1\end{pmatrix}.
    \end{align*}
    \[
    \]
With the block decomposition \(M=\begin{pmatrix}A&b\\ b'&c\end{pmatrix}\) of this Hessian, where 
\[
A=\Sigma_{P_{N}}+\mu_{P_{N}}\mu_{P_{N}}',\,b=-\mu_{P_{N}},\,\hbox{and } c=1,
\] the Schur complement of block \(c\) of matrix $M$ is \(A-bc^{-1}b'=\Sigma_{P_{N}}\), which is positive definite by assumption; since \(c>0\) it follows from the Schur-complement criterion that \(M\succ0\). %
With \(t:=\tfrac{2a+\gamma}{2}\) and \(z:=(\mathbf{x},t)'\), we have  
\[
S_\gamma(\mathbf{x},a)^2=\mathbf{x}'\Sigma_{P_{S_0}}\mathbf{x}+(\mathbf{x}'\mu_{P_{S_0}}-t)^2
= z'\begin{pmatrix}\Sigma_{P_{S_0}}+\mu_{P_{S_0}}\mu_{P_{S_0}}' & -\mu_{P_{S_0}}\\[4pt]-\mu_{P_{S_0}}' & 1\end{pmatrix}z=:z'Bz,
\]
where $B \succ 0$ from the Schur-complement criterion.
Since $\x$ is restricted to a probability simplex \(\Delta\), no two distinct points in the domain lie on the same positive ray through the origin. Therefore \(S_\gamma(\mathbf{x},a)=\sqrt{z'Bz}\), being a norm, is strictly convex on the domain, though convexity suffices for our development. We hence have $$ V(q, \x, a)=\left( r(q) \|\x\|+S_\gamma(\mathbf{x},a)\right)^2$$ is convex in $(\x,a)$ as a result of an increasing convex function composed with a convex function being convex. For $q<1$, we can now conclude that $h$ is strongly convex on $(\x,a)$, as
\[
\nabla_{(\x,a)}^2 h(q,\x,a)  \succeq (1-q) \nabla_{(\x,a)}^2\EP{P_N}{\psi_{a,\gamma}(\mathbf{x}'\mathbf{R})} = (1-q)M.
\]
As a result, for a fixed $q<1$, $h$ is $((1-q)\lambda_M)$-strongly convex over $(\x,a)$, with $\lambda_M$ the smallest eigenvalue of $M$.

\item Since we assume $q_0+\varepsilon<1$, $J$ is $((1-q_0-\varepsilon)\lambda_M)$-strongly convex as the supremum of strongly convex functions.  $J$ is $K_1$-Lipschitz, since it has a subgradient bounded by $K_1$ for every $(\x,a)$.
\item      Strongly concavity of $h$ implies that $\argmax_q h(q, \x, a)$ is a singleton, say containing $q^*(\x,a)$. We know from Danskin's Theorem, $\nabla J(\x,a)=\nabla_{(\x,a)} h(q^*(\x,a),\x,a)$. Since $\nabla h$ is Lipschitz, it suffices to prove $q^*$ is Lipschitz. 
Define 
$$F(q,\x,a)= h_q(q,\x,a).$$ 
Since $\frac{\partial^2}{\partial q^2} h$ is a continuous functions on a compact domain, and $h$ is strongly concave, we have $ F_q=h_{qq} \leq -\alpha<0$ for a suitable positive constant $\alpha$. 

Recall that
\[
 \left\|\frac{\partial}{\partial (\x,a)} F\right\| = \left\|\frac{\partial^2}{\partial q \partial (\x,a)} h\right\| \leq K_3.
\]

If we had $F(q^*(\x, a),\x, a)=0$ for all $\x,a$, it is straightforward to see that $q^*$ is a $ \frac{K_3}{\alpha}-$Lipschitz function. However, the above first-order condition need not hold since $q$ lies in a bounded interval. Nevertheless, we will show that this holds by using the KKT condition, which accounts for the interval constraint. Consider $(\x_1, a_1)$ and $(\x_2, a_2)$, with $q^*(\x_1, a_1) < q^*(\x_2, a_2)$. Since $q^*(\x_1, a_1)$ may equal $q_0-\varepsilon$, and $q^*(\x_2, a_2)$ may equal $q_0+\varepsilon$, the KKT condition implies $$F(q^*(\x_1, a_1),\x_1, a_1)-F(q^*(\x_2, a_2),\x_2, a_2)\leq 0.$$
Therefore, using the mean value theorem, we have, 
\begin{align*}
    0&\geq F(q^*(\x_1, a_1),\x_1, a_1)-F(q^*(\x_2, a_2),\x_2, a_2) \\
    &=(F(q^*(\x_1, a_1),\x_1, a_1)-F(q^*(\x_1, a_1),\x_2, a_2))\\
    &\quad+(F(q^*(\x_1, a_1),\x_2, a_2)-F(q^*(\x_2, a_2),\x_2, a_2))\\
    &= \frac{\partial F}{\partial (\x,a)}(q^*(\x_1, a_1),\x_0,a_0) \cdot ((\x_1, a_1)-(\x_2, a_2)) \\
    & \quad+ F_q(q_0,\x_2, a_2) \cdot (q^*(\x_1, a_1)-q^*(\x_2, a_2)) \\
    & \geq -K_3 \left\|(\x_1, a_1)-(\x_2, a_2) \right\|+\alpha ((q^*(\x_2, a_2)-q^*(\x_1, a_1))),
\end{align*}
where $(\x_0,a_0)$ is a point lying on the segment between $(\x_1,a_1)$ and $(\x_2,a_2)$ and 
\[
q_0 \in (q^*(\x_1, a_1),  q^*(\x_2, a_2)).
\]
As a result, 
\[
   \left\vert q^*(\x_1, a_1)-q^*(\x_2, a_2)\right\vert \leq \frac{K_3}{\alpha} \left\|(\x_1, a_1)-(\x_2, a_2) \right\|
\]

Therefore, $q^*$ is a $ \frac{K_3}{\alpha}-$Lipschitz function. Recall $\nabla J(\x,a)=\nabla_{(\x,a)} h(q^*(\x,a),\x,a)$. By using a mean value theorem for vector-valued maps (see Theorem 12.9 in \citet{Apostol:105425}),
\begin{align*}
       &\left\| \nabla J(\x_1, a_1)-\nabla J(\x_2, a_2)\right\| \\
       &\leq \left\| \nabla_{(\x,a)} h(q^*(\x_1,a_1),\x_1,a_1)-\nabla_{(\x,a)} h(q^*(\x_2,a_2),\x_1,a_1)\right\|  \\
       & \quad \quad +\left\| \nabla_{(\x,a)} h(q^*(\x_2,a_2),\x_1,a_1)-\nabla_{(\x,a)} h(q^*(\x_2,a_2),\x_2,a_2)\right\| \\
       & \leq K_3 \cdot \frac{K_3}{\alpha}  \left\|(\x_1, a_1)-(\x_2, a_2) \right\| + K_2 \left\|(\x_1, a_1)-(\x_2, a_2) \right\| \\
       & \leq \left( \frac{K_3^2}{\alpha}  + K_2 \right) \left\|(\x_1, a_1)-(\x_2, a_2) \right\| ,
\end{align*}
where we recall the definition of 
\[K_2 :=  \max_{(q,\x,a) \in \Omega}\left\lVert\frac{\partial^2}{\partial (\x,a)^2} h \right\rVert_2.
\]
As a result, $\nabla J$ is $\left( \frac{K_3^2}{\alpha}  + K_2 \right)$-Lipschitz. $J$ is $\left(  \frac{K_3^2}{\alpha}  + K_2 \right)$-smooth.

\end{enumerate}

\end{proof}

\subsubsection{Proof of Theorem \ref{fastConvergeVar}}

\begin{proof}
    We first show there exists $\varepsilon_0$ such that $h$ is strongly concave in $q$ on $ [q_0-\varepsilon_0,q_0+\varepsilon_0]$.
    \[
\begin{aligned}
h_{qq}(q,\x,a)
&=2V_q(q,\x,a)+qV_{qq}(q,\x,a)\\
&=4r'(q)\Big(r(q)\|\x\|^2+\|\x\|\,S_{\gamma}(\x,a)\Big)
+2qr''(q)\Big(r(q)\|\x\|^2+\|\x\|\,S_{\gamma}(\x,a)\Big)\\
&\quad +2q\big(r'(q)\big)^2\|\x\|^2.
\end{aligned}
\]
Let
$$K_5:=\max_{q \in [0,1]}r(q),\, \hbox{and }K_6:=\max_{\x \in \Delta, a \in I}S_\gamma(\x,a).$$
Define an auxiliary function 
$$\bar{h}_{qq}(q):=4\left\vert r'(q) \right \vert \left(K_5+K_6 \right)
+2qr''(q)r(q)/d+2q\big(r'(q)\big)^2.$$
Since $r''$ is continuous, there exists an $\varepsilon_1>0$, such that $\bar{h}_{qq}(q) \geq h_{qq}(q,\x,a)$ on $$q \in [q_0-\varepsilon_1,q_0+\varepsilon_1]$$ for all $\x \in \Delta, a \in I$.

Since $r'(q_0)=0$, we have 
   \[
   \bar{h}_{qq}(q_0)=2 q_0 r(q_0) r''(q_0)/d<0.
   \]
   Since $\bar{h}_{qq}$ is continuous, there exists $\varepsilon_2>0$ such that $\bar{h}_{qq}(q)\leq -K_4 <0$ on $q \in [q_0-\varepsilon_2,q_0+\varepsilon_2]$. Taking $\varepsilon_0=\varepsilon_1 \wedge \varepsilon_2$ shows $h$ is strongly concave in $q$ on $ [q_0-\varepsilon_0,q_0+\varepsilon_0]$. Combining Lemma \ref{propertieshJ} ii., iii., and Theorem \ref{fastConverge} concludes the proof.
\end{proof}

\subsubsection{Proof of Theorem \ref{oneOverD}}

\begin{proof}
We first observe the fact that the minimum $l_2$ norm vector in the simplex is clearly $\x=(1/d) \cdot \mathbf{1}$, corresponding to the equally weighted portfolio.
\begin{align*}
    \frac{h(q,\x,a)}{C^2}
    &= \frac{(1-q)\EP{P_{N}}{\psi_{a,\gamma}(\x'\R)}}{C^2} + q \left( r_{\text{base}}(q)^2 \|\x\|^2 + \frac{2r_{\text{base}}(q)\|\x\| S_\gamma(\x,a)}{C} \right.\\
    & \hspace{6cm}\left.+ \frac{S_\gamma(\x,a)^2 - a\gamma - \gamma^2/4}{C^2} \right)\\
    & \rightrightarrows q \cdot r_{\text{base}}(q)^2 \|\x\|^2,
\end{align*}
as  $C \to \infty$, where uniform convergence is over $q,\x,a$, as a result of the boundedness of the other terms on the numerators.

We let $C>1, K_7<\infty$ such that
$$\left\vert\frac{(1-q)\EP{P_{N}}{\psi_{a,\gamma}(\x'\R)}}{C^2} + q \left(  \frac{2r_{\text{base}}(q)\|\x\| S_\gamma(\x,a)}{C} + \frac{S_\gamma(\x,a)^2 - a\gamma - \gamma^2/4}{C^2} \right)\right\vert \leq \frac{K_7}{C}.$$
Due to continuity of $r_{\text{base}}$, we denote 
\[
\max_{q \in [q_0-\varepsilon, q_0+\varepsilon]} q \cdot r_{base}(q)^2 = K_8>0,
\] 
with $q^*$ being the maximum point. Denote the nonempty set $I_q=\left\{q:q \cdot r_{\text{base}}(q)^2 \ge \frac{K_8}{2}\right\}$.
When $C>\frac{4K_7n}{K_8}\vee 1$, it is easy to see 
\[
(\x_C^*, a_C^*)=\argmin_{\x \in \Delta,a \in I}\;\max_{q \in I_q} \frac{h(q,\x,a)}{C^2},
\]
as for a fixed $\x,a$ and $q_0 \notin I_q$, 
\[\frac{h(q^*,\x,a)}{C^2}-\frac{h(q_0,\x,a)}{C^2} \geq \frac{K_8}{2} \cdot \frac{1}{d}-2\frac{K_7}{C}\geq0.\] 

Let $\x^*=(\frac{1}{d},\ldots,\frac{1}{d})$ and $\mathbf{1}'\boldsymbol{\eta}=0$. For any $q \in I_q$, $a,a' \in I$ if $\|\boldsymbol{\eta}\| > \sqrt{\frac{4K_7}{CK_8}}$,
\[
\frac{h(q,\x^*+\boldsymbol{\eta},a')}{C^2}-\frac{h(q,\x,a^*)}{C^2} \geq \frac{K_8}{2}\|\boldsymbol{\eta}\|^2-2\frac{K_7}{C}>0.
\]
Therefore, $\x^*+\boldsymbol{\eta}$ cannot be the minimum point. Therefore, as $C$ increases, the minimum point $$\x_C^* \in B_{\sqrt{\frac{4K_7}{CK_8}}}(\x^*),$$ is guaranteed to lie in a shrinking ball around the equal-weighted portfolio, with radius $\sqrt{\frac{4K_7}{CK_8}}$ decreasing to 0.

\end{proof}

\section{Mean-CVaR Disutility function}
Variance, as a risk measure, has limitations; for example, it treats departures from the mean symmetrically, is not positively homogeneous, and is not shift-equivariant. In insurance regulation, for this reason and others, the preference is for the coherent risk measure, conditional value-at-risk (CVaR), also known as the Tail Value-at-Risk (T-VaR) and the Condition Tail Expectation (CTE).  In this section, we consider an alternative formulation of the problem from the previous section, with CVaR replacing variance as the risk measure.

There are a few different definitions of CVaR; for our development, the following variational definition of \citet{rockafellar2000optimization} is the most pertinent: 

\begin{definition}
Let $X$ be a random variable with a finite mean. Then for $p\in(0,1)$, the $p$-level conditional value at risk of $X$ is denoted by $\CVaR{X}$ and given by,
\begin{equation}\label{TVaRdefn}
    \CVaR{X} = \inf_{\tau \in \mathbb{R}} \left\{ \tau + \frac{1}{1-p} \E{\max(X - \tau, 0)} \right\}.
\end{equation}
\end{definition}

Again, the true distribution of the return $\R$ is constrained to a two-component mixture structure, with a component distribution and the weight unknown. The investor's mean-CVaR disutility function is a weighted sum of the expected loss and its Conditional Value-at-Risk at a confidence level $p \in (0,1)$:
\begin{equation}
    \E{-\x'\R} + \rho \CVaR{-\x'\R},
\end{equation}
where $\rho > 0$ is a risk-aversion parameter. While the true distribution of returns, denoted by $P$, is unknown, investor beliefs restrict it to an ambiguity set $\calS$. The investor's robust portfolio design problem reduces to solving the following optimization problem:
\begin{equation}
\tag{Opt 2}
\label{mean-cvar}
    \inf_{\mathbf{x} : \x \in \Delta} \sup_{P \in \calS} \left( \EP{P}{-\x'\R} + \rho \CVaR{-\x'\R} \right).
\end{equation}
Similar to Section 4, we define the ambiguity set $\calS$ as follows.
\begin{equation}
\label{ambi_CVaR}
    \calS = \{(1-q)P_N + qP_S \mid q \in [q_0-\epsilon, q_0+\epsilon], \, P_S \in \mathbb{B}_{W_1}(P_{S_0}, r(q)) \}.
\end{equation}
The only difference from \eqref{ambiset} is the Wasserstein-1 ball of radius $r(q)$ around $P_{S_0}$. We assume that the asset returns under both distributions $P_N$ and $P_{S_0}$ have finite first moments, i.e., $\EP{P_{N}}{\|\mathbf{R}\|} < \infty$ and $\EP{P_{S_0}}{\|\mathbf{R}\|} < \infty$. Since $-\x'\R$ has a finite mean for $x \in \Delta$ under any distribution in $\calS$, its mean-CVaR disutility is well defined for such distributions as well. The main result of this section is the following:

\begin{Theorem}
\label{mean_cvar_minimax_reform}
Define the loss function $\ell_{\x,\tau}(\mathbf{R}) := -\x'\R + \frac{\rho}{1-p}\max(-\x'\R - \tau, 0)$. 
    \eqref{mean-cvar} is equivalent to the following:

\begin{equation}
    \label{minimax_cvar}
    \min_{\mathbf{x} \in \Delta, \tau \in I} J(\x,\tau),
\end{equation}
where 
\begin{equation}\label{eqn:cvarJ}
J(\x,\tau)=\left\{ \rho\tau + \sup_{q \in [q_0\pm\epsilon]} \left( (1-q)\EP{P_{N}}{\ell_{\x,\tau}(\mathbf{R})} + q\left[\EP{P_{S_0}}{\ell_{\x,\tau}(\mathbf{R})} + r(q)(1+\frac{\rho}{1-p})\|\mathbf{x}\|_\infty\right] \right) \right\}.    
\end{equation}

Moreover, $J$ is convex in $(\mathbf{x}, \tau)$.

\qed
\end{Theorem}

A proof of the above theorem is provided in the next subsection. The objective function $J$ in its statement is a non-differentiable convex function; hence, for its minimization, we propose a projected subgradient descent algorithm as given in Algorithm \ref{algo2} below.
\begin{algorithm}
\caption{Projected Subgradient Descent for Robust Mean-CVaR \label{algo2}}
\begin{algorithmic}[1]
\State \textbf{Initialize:} Choose $(\mathbf{x}_0,\tau_0) \in \Delta \times \Real$, the number of iterations $T>0$ and step sizes $\eta_k > 0$ for $k=0,1,\ldots,T$.
\For{$k=0, 1, \dots, T$}
    \State Find a worst-case probability ${q}^*_k$.
    \State Compute the subgradient components of $J$ at $(\x_k. \tau_k)$: $\mathbf{g}_{k, \mathbf{x}}, g_{k, \tau}$.
    \State Update the variables:
    \begin{align*}
        \mathbf{x}_{k+} &\leftarrow \mathbf{x}_k - \eta_k \mathbf{g}_{k, \mathbf{x}} \\
        \tau_{k+1} &\leftarrow \tau_k - \eta_k g_{k, \tau}
    \end{align*}
    \State Project $\mathbf{x}_{k+1} \leftarrow \text{Proj}_{\Delta}(\mathbf{x}_{k+})$.
\EndFor
\State \textbf{Return} $(\x_1,\tau_1),\ldots,(\mathbf{x}_T,\tau_T)$.

\end{algorithmic}
\end{algorithm}

At each iteration $k$, given $(\mathbf{x}_k, \tau_k)$, we first find the $q_k^*$ achieving the supremum in \eqref{eqn:cvarJ}. By Danskin's Theorem, a subgradient $\mathbf{g}_k \in \partial J(\mathbf{x}_k, \tau_k)$ is given by a subgradient of the inner function evaluated at $q_k^*$. We have the following Proposition for a subgradient of $J$. 

\begin{Proposition}
\label{subgradientCVAR}
 A subgradient, $\mathbf{g}_k$,  of $J$ at $(\x_k. \tau_k)$ has the following components:
\begin{multline*}
   \mathbf{g}_{k, \mathbf{x}} = -(1-q_k^*)\left[\EP{P_N}{\mathbf{R}} + \left(\frac{\rho}{1-p}\right)\EP{P_N}{\mathbf{R} \cdot \mathbf{1}_{-\mathbf{x}'\mathbf{R} > \tau}}\right] \\
   + q_k^* \left[-\left(\EP{P_{S_0}}{\mathbf{R}} + \left(\frac{\rho}{1-p}\right)\EP{P_{S_0}}{\mathbf{R} \cdot \mathbf{1}_{-\mathbf{x}'\mathbf{R} > \tau}}\right) +r(q)(1+\frac{\rho}{1-p})g_{\infty}(\x)\right],
\end{multline*} 
where $g_{\infty}(\x) \in \partial \|\x\|_{\infty}$, and  
\begin{equation*}
g_{k, \tau} = \rho - (1-q_k^*)\left(\frac{\rho}{1-p}\right)P_N(-\mathbf{x}'\mathbf{R} > \tau) - q_k^*\left(\frac{\rho}{1-p}\right)P_{S_0}(-\mathbf{x}'\mathbf{R} > \tau).
\end{equation*}
\qed
\end{Proposition}

Note that in the statement of Proposition \ref{subgradientCVAR}, the set of the subgradients of $\|\x\|_{\infty}$ is denoted by $\partial \|\x\|_{\infty}$, and is given by 
\[
\partial \|\x\|_{\infty}:=\text{conv}(\{\text{sign}(\x_{i})e_i :|\x_i|=\|\x\|_\infty\}),
\]
where $e_i$ is the $i$-th standard basis vector of $\Real^n$ and for a set $A$, $\text{conv}(A)$ denotes its convex hull. The following theorem provides a convergence guarantee for Algorithm \ref{algo2}. 

\begin{Theorem}
   Let $K_6=\EP{P_{N}}{\|\mathbf{R}\|}, K_7=\EP{P_{S_0}}{\|\mathbf{R}\|}, K_8=\max_{q \in [0,1]} r(q)$ and 
   \[
   K_9=\sqrt{\left[\left(1+\frac{\rho}{1-p}\right)\left((K_6 \vee K_7)+K_8\right)\right]^2+\rho^2\left(\frac{p}{1-p}\vee 1\right)^2}.\] Let $R$ be such that $\|(\x_0,\tau_0)-(\x^*,\tau^*)\| \leq R$. If we use constant step size $\eta_k=\frac{R}{K_9 \sqrt{T}}$ in Algorithm \ref{algo2}, we have 
\[J\left(\frac{1}{T}\sum_{i=1}^{T}(\x_i,\tau_i)\right)-J(\x^*,\tau^*) \leq \frac{RK_9}{\sqrt{T}}.\]
   In other words, Algorithm \ref{algo2} has a convergence rate of at least $O(1/\sqrt{T})$.
\end{Theorem}

\begin{proof}
     \[
     \|\mathbf{g}_{k, \mathbf{x}}\| \leq \left(1+\frac{\rho}{1-p}\right)\left((K_6 \vee K_7)+K_8\right),\, |g_{k, \tau}| \leq \rho \vee \left(\frac{\rho}{1-p}-\rho\right)=\rho\cdot\left(\frac{p}{1-p}\vee 1\right),
     \]
     where we applied Jensen's inequality to the 2-norm. 
     As a result, $J$ is $K_9$-Lipschitz. 
     The proof is concluded by Theorem \ref{slowConverge}.
 
\end{proof}

Finally, Theorem \ref{oneOverD-CVaR} shows that as the ambiguity becomes large, $\x^*$ approaches $(\frac{1}{d},\ldots,\frac{1}{d})$, the equal-weighted portfolio. We omit the proof for its similarity with Theorem \ref{oneOverD} in Section 2.

\begin{Theorem}
\label{oneOverD-CVaR}
Let $(\x_C^*, a_C^*)$ be a minimizer of $J$ with the ambiguity radius in \eqref{ambi_CVaR} defined by $r(q) = C \cdot r_{\text{base}}(q)$, where $C>0$ is a scaling parameter and $r_{\text{base}}$ is a fixed, non-negative continuous function on the interval $[0,1]$,  that is not identically zero on $[q_0-\varepsilon, q_0+\varepsilon]$. Then,  as $C \to \infty$, $\x_C^*$ converges to the equal weighted portfolio $\frac{1}{d} \cdot \mathbf{1}$.
\end{Theorem}

\subsection{Proofs of Results}
To prove Theorem \ref{mean_cvar_minimax_reform}, we need the following two lemmas. 

\begin{lemma}
The mean-CVaR problem \eqref{mean-cvar} has the following equivalent characterization:    
\begin{multline}
\label{exchangeCVAR}
     \inf_{\x \in \Delta} \sup_{P \in \calS} \left( \EP{P}{-\x'\R} + \rho \CVaR{-\x'\R} \right) \\=   \inf_{\mathbf{x} \in \Delta, \tau \in I} \left\{ \rho\tau + \sup_{P \in \calS} \EP{P}{-\x'\R + \frac{\rho}{1-p}\max(-\x'\R - \tau, 0)} \right\},
\end{multline}
where $I$ is a compact interval independent of $\x$ and $P$.
\end{lemma}
\begin{proof}
From the definition \eqref{TVaRdefn}:
\[
      \sup_{P \in \calS} \left( \EP{P}{-\x'\R} + \rho \CVaR{-\x'\R} \right) =   \sup_{P \in \calS} \inf_{\tau \in  \Real} \left\{ \rho\tau + \EP{P}{-\x'\R + \frac{\rho}{1-p}\max(-\x'\R - \tau, 0)} \right\},
\]
Denote the loss random variable $L:=-\x'\R$. It is well known that for $p \in (0,1)$, the set of $\tau$  achieving the infimum equals $[\vec{F}_L(1-p),\cev{F}_L(1-p)]$, where 
\[
\cev{F}_L(1-p):=\sup\{x:F_L(x) \leq 1-p\},\, \hbox{and }\, \vec{F}_L(1-p) := \inf\{x:F_L(x) \geq 1-p\},
\]
are the right and left generalized inverses of $F_L$, the cdf of $L$, which depends on $\x$ and $P$.
The right and left generalized inverses are non-decreasing and right-continuous and left-continuous, respectively. 
 We first show that the inner optimization can be restricted to a compact interval $I$ free of $P\in \calS$ and $\mathbf{x}\in\Delta$.
It suffices to show an $I$ containing the $[\vec{F}_L(1-p),\cev{F}_L(1-p)]$ for all $P\in \calS$ and $\mathbf{x}\in\Delta$.  As f is continuous, let $M$ be the maximum of $f$ on $[q_0-\varepsilon,q_0+\varepsilon]$. Then for any $P_{S}$ in $\mathbb{B}_{W_1}(P_{S_0}, M)$ and $\mathbf{x}\in\Delta$, if we denote $\pi$ to be the optimal 1-Wasserstein coupling between $P_S$ and $P_{S_0}$ and $(\R,\R_0) \sim \pi$, we have
\begin{multline*}
    \EP{P_{S}}{\|\R\|_1} = \EP{\pi}{\|\R_0+\R-\R_0\|_1} \leq \EP{\pi}{\|\R_0\|_1}+\EP{\pi}{\|\R-\R_0\|_1}
    \leq \EP{P_{S_0}}{\|\R\|_1}+M,
\end{multline*}

and 
\[
\EP{P_{S}}{|\x'\R|} \leq \|\x\|_\infty\,\EP{P_{S}}{\|\R\|_1} \leq \EP{P_{S_0}}{\|\R\|_1}+M. 
\]
Similarly, 
$$
\EP{P_{N}}{|\x'\R|} \leq \|\x\|_\infty \,\EP{P_{N}}{\|\R\|_1} \leq \EP{P_{N}}{\|\R\|_1}.
$$ 
Since,  any $P$ in $\calS$ is a mixture of $P_{N}$ and a $P_{S}$ in $\mathbb{B}_{W_1}(P_{S_0}, M)$, with mixture weight $q$ in $[q_0-\varepsilon,q_0+\varepsilon]$, combining the above statements, we have,
$$
\EP{P}{\left\vert\x'\R\right\vert} = (1-q)\EP{P_{N}}{|\x'\R|}+q\EP{P_{S}}{|\x'\R|} \leq (\EP{P_{S_0}}{\|\R\|_1}+M)\vee 
 \EP{P_{N}}{\|\R\|_1}.
$$

For any $t \in [\vec{F}_L(1-p),\cev{F}_L(1-p)]$, 
\[
|t| \leq \left(\frac{1}{1-p} \vee \frac{1}{p}\right) \cdot \left[(\EP{P_{S_0}}{\|\R\|_1}+M)\vee 
 \EP{P_{N}}{\|\R\|_1}\right].
 \]
The above shows the existence of our desired $I$.

Since $I$ is compact, $\rho\tau + \EP{P}{-\x'\R + \frac{\rho}{1-p}\max(-\x'\R - \tau, 0)}$ is convex and continuous in $\tau$, and concave (linear) and continuous in $P$ with respect to the 1-Wasserstein metric topology (see Theorem 7.12 in \citet{villani2003topics} for implications of $W_1$ convergence), from Sion's Minimax Theorem (see Theorem \ref{sion} in the Appendix), we can swap the min and sup operators, and conclude the proof.
\end{proof}
The core of the problem is to solve the inner supremum $\sup_{P \in \calS} \EP{P}{\ell_{\x,\tau}(\mathbf{R})}$.

\begin{lemma}
    $\ell_{\x,\tau}$ is Lipschitz with respect to $L_1$ norm, and has Lipschitz constant $$\left(1 + \frac{\rho}{1-p}\right) \|\mathbf{x}\|_\infty$$
\end{lemma}

\begin{proof}
The dual norm of $L_1$-norm is the $L_\infty$-norm. The Lipschitz constant is the maximum dual norm of the gradient of the loss function.  The gradient of the loss function with respect to the vector $\mathbf{R}$ is:
\begin{equation}
\nabla_{\mathbf{R}} \ell_{\x,\tau}(\mathbf{R}) =
\begin{cases}
    -\mathbf{x} & \text{if } -\mathbf{x}'\mathbf{R} - \tau < 0 \\
    -\left(1 + \frac{\rho}{1-p}\right)\mathbf{x} & \text{if } -\mathbf{x}'\mathbf{R} - \tau > 0
\end{cases}
\end{equation}
The Lipschitz constant is the maximum $L_\infty$-norm of this gradient:
\begin{align*}
    \text{Lip}(\ell_{\x,\tau}) &= \max \left( \|-\mathbf{x}\|_\infty, \left\|-\left(1 + \frac{\rho}{1-p}\right)\mathbf{x}\right\|_\infty \right) = \left(1 + \frac{\rho}{1-p}\right) \|\mathbf{x}\|_\infty.
\end{align*}

\end{proof}

\begin{lemma}
For a fixed $\mathbf{x}, \tau, q$, the worst-case expectation of the loss over the Wasserstein-1 ball:
    \begin{equation}
    \label{worstW1}
    V(q, \tau, \mathbf{x}) := \sup_{P_S \in \mathbb{B}_{W_1}(P_{S_0}, r(q))} \EP{P_{S}}{\ell_{\x,\tau}(\mathbf{R})}=\EP{P_{S_0}}{\ell_{\x,\tau}(\mathbf{R})} + r(q) \cdot \text{Lip}(\ell_{\x,\tau}).
    \end{equation}
\end{lemma}

\begin{proof}
    From \eqref{duality}, \begin{equation*}
    V(q, \tau, \mathbf{x}) = \inf_{\lambda \ge 0} \left\{ \lambda r(q) + \EP{P_{S_0}}{\sup_{\r} \left( \ell_{\x,\tau}(\r) - \lambda \|\r - \mathbf{R}\|_1 \right)} \right\},
\end{equation*}
 where $\R \sim P_{S_0}$. Note that $\sup_{\r} \{ \ell_{\x,\tau}(\r) - \lambda \|\mathbf{r} - \mathbf{R}\|_1 \} = \ell_{\x,\tau}(\mathbf{R})$ when $\lambda \geq \text{Lip}(\ell_{\x,\tau})$, otherwise the $\sup$ is infinity. Therefore, the dual problem becomes:
\begin{equation*}
    V(q, \tau, \mathbf{x}) = \inf_{\lambda \ge \text{Lip}(\ell_{\x,\tau})} \left\{ \lambda r(q) + \EP{P_{S_0}}{\ell_{\x,\tau}(\mathbf{R})} \right\}=\EP{P_{S_0}}{\ell_{\x,\tau}(\mathbf{R})} + r(q) \cdot \text{Lip}(\ell_{\x,\tau}).
\end{equation*}

\end{proof}
\begin{Remark}
    The above lemma is not surprising, as by the Kantorovich-Rubinstein duality for the Wasserstein-1 distance, 
    $$
    d_{W_1}(P_{S_0}, P_S) \cdot \text{Lip}(\ell_{\x,\tau}) \ge \EP{P_{S}}{\ell_{\x,\tau}(\mathbf{R})}-\EP{P_{S_0}}{\ell_{\x,\tau}(\mathbf{R})}
    $$. 
    
    As a result,
\begin{equation*}
    V(q, \tau, \mathbf{x}) \leq \EP{P_{S_0}}{\ell_{\x,\tau}(\mathbf{R})} + r(q) \cdot \text{Lip}(\ell_{\x,\tau}).
\end{equation*}

\end{Remark}

Now we return to the proof of Theorem \ref{mean_cvar_minimax_reform}:

\begin{proof}

From \eqref{exchangeCVAR} and \eqref{worstW1}
    \begin{align*}
\label{exchangeCVAR}
     & \inf_{\x \in \Delta} \sup_{P \in \calS} \left( \EP{P}{-\x'\R} + \rho \CVaR{-\x'\R} \right) \\
     &= \inf_{\mathbf{x} \in \Delta, \tau \in I} \left\{ \rho\tau + \sup_{P \in \calS} \EP{P}{-\x'\R + \frac{\rho}{1-p}\max(-\x'\R - \tau, 0)} \right\}\\
     &=    \inf_{\mathbf{x} \in \Delta, \tau \in I} \left\{ \rho\tau + \sup_{q \in [q_0\pm\epsilon]} \left( (1-q)\EP{P_{N}}{\ell_{\x,\tau}(\mathbf{R})} + q\left[\EP{P_{S_0}}{\ell_{\x,\tau}(\mathbf{R})} + r(q)(1+\frac{\rho}{1-p})\|\mathbf{x}\|_\infty\right] \right) \right\}\\
     &=  \inf_{\mathbf{x} \in \Delta, \tau \in I} J(\x,\tau).
\end{align*}
Since $\ell_{\x,\tau}(\mathbf{R}) = -\x'\R + \frac{\rho}{1-p}\max(-\x'\R - \tau, 0)$ is convex in $(\mathbf{x}, \tau)$, the convexity of J comes from the fact that any norm is convex and the supremum of convex function is convex.

By Berge's maximum theorem, $J$ is continuous since $[q_0-\varepsilon, q_0+\varepsilon] $ is a compact set, and the objective function is continuous in $(q,\x,\tau)$. We can change the $\inf$ to $\min$ since $J$ is continuous on a compact set.

\end{proof}

\begin{Remark}
    An easier way to see the continuity of $J$ is the following argument. Note that $J$ is convex in an open ball containing $\Delta \times I$. By Theorem 10.1 of \citet{rockafellar1970convex} (a convex function on $\Real^n$ is continuous relative to the relative interior of its domain), $J$ is continuous on $\Delta \times I$.
\end{Remark}

\section{Simulation Experiments}

We begin by describing the setup for generating our simulated data. We consider 1,000 i.i.d. samples of returns for 10 assets ($d=10$) and set the true probability of the stress regime to 3\% ($q=3\%$).  In the normal regime, the return of the $i$-th asset is decomposable into an idiosyncratic component $N(\mu=0.03i,\sigma=0.025i)$ and a systematic component $N(\mu=0,\sigma=0.02)$ common to all assets, which is the simulation setup in section 7.2 of \citet{mohajerin2018data}. In other words, the returns of the assets under $P_N$ follow a multivariate $N_d(\boldsymbol{\mu},\Sigma)$, where with $\delta_{\cdot}$ denoting the Kronecker delta function, 
\[
\mu_i=0.03i, \hbox{ and }\sigma_{ij}=0.02^2+0.025^2i^2\delta_{i,j}.
\]
In the stress regime, the returns of the assets follow a heavy-tailed multivariate $t$ distribution with five degrees of freedom, $t_{d}(\boldsymbol{\mu},\Sigma,\nu=5)$, where 
\[
\mu_i=-0.05(i+1), \hbox{ and }\sigma_{ij}=(0.1+0.03i)\cdot (0.1+0.03j) \cdot \left(0.7 +0.3\delta_{ij}\right),
\]
resulting in a high base correlation of 0.7 between assets, and a standard deviation increasing linearly with the index. It also captures the stylized fact that assets with higher returns and variance in the normal regime (e.g., smaller companies) tend to react more severely in the stress regime. Another important aspect of our setup is that the assets are tail-independent in the normal regime but tail-dependent in the stress regime, with a tail dependence coefficient of 0.43. Figure \ref{fig:visual} is a visualization of these distributions via a scatter plot of 1,000 samples of Asset 1 and Asset 2 returns.
   \begin{figure}[htbp]
        \centering %
        \includegraphics[width=0.7\textwidth]{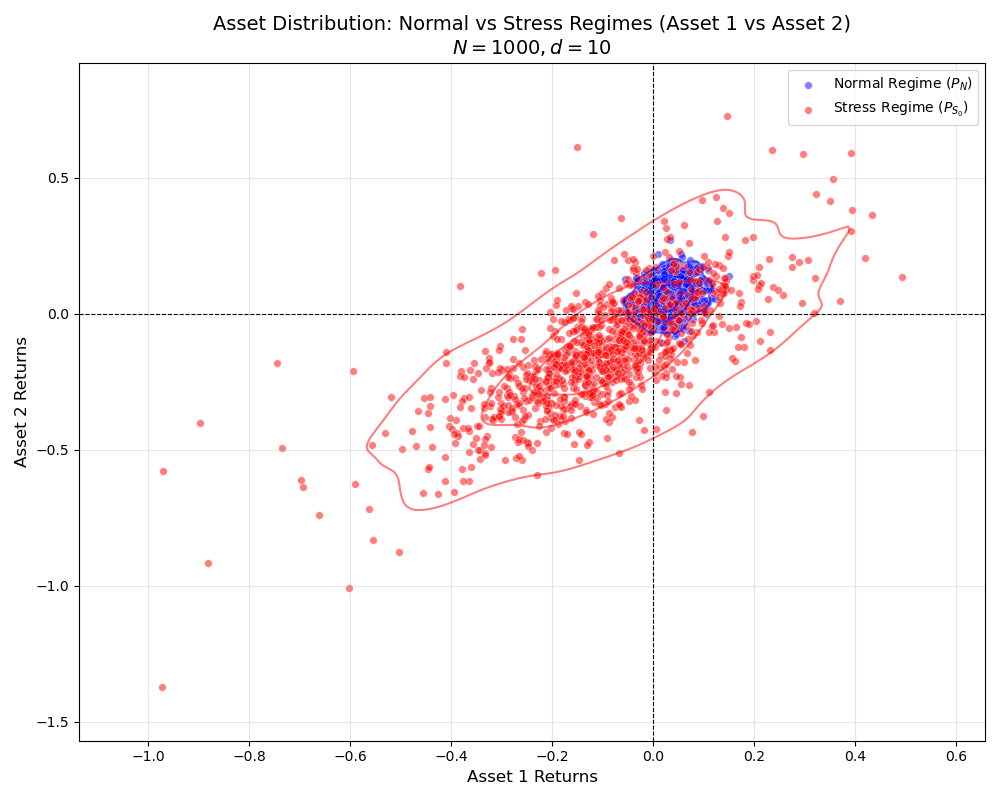} %
        \caption{Scatter plot of 1,000 samples of Asset 1 and Asset 2 returns, with the contours of the kernel density estimate of the joint density under the two regimes. } %
        \label{fig:visual} %
    \end{figure}
    
We now describe the different choices of $\mathcal{S}$ we consider in our experiments. We allow for $\varepsilon \in \{0,0.01,0.02,0.03\}$, and the radius $r(\cdot)$ is restricted to the functional form 
\[
r(q)=c \cdot q^{\alpha-1} (1-q)^{\beta-1},
\]
where to achieve the mode of $q_0$, we choose  
\[
\alpha=M \cdot q_0 + 1,\,\beta=M \cdot (1-q_0) + 1. 
\] 
In our experiments, we choose $M=10$. We deliberately choose some of the intervals $[q_0-\varepsilon,q_0+\varepsilon]$ to contain $0$, to test the empirical performance of our theory when some assumptions are violated.

\subsection{Mean-Var Disutility} %

We select $\gamma=0.1$ to indicate high risk aversion.
To analyze the convergence rate of our projected subgradient descent algorithm, we choose a fixed step-size ($\eta$) of $0.001$, and choose $c$, $\varepsilon$, and $q_0$ to equal 0.1, 0.03, and 0.024, respectively. %
We plot the suboptimality gap, $J(x_k, a_k) - J^*$, on a logarithmic scale against the iteration number $k$. Since $0 \in [q_0-\varepsilon,q_0+\varepsilon]$, as a result, $h$ is not strongly concave on $[q_0-\varepsilon,q_0+\varepsilon]$. However, in Figure \ref{fig:convergence1}, the convergence rate is still faster than an exponential rate, as the graph lies under a straight line.

    \begin{figure}[htbp]
        \centering %
        \includegraphics[width=0.7\textwidth]{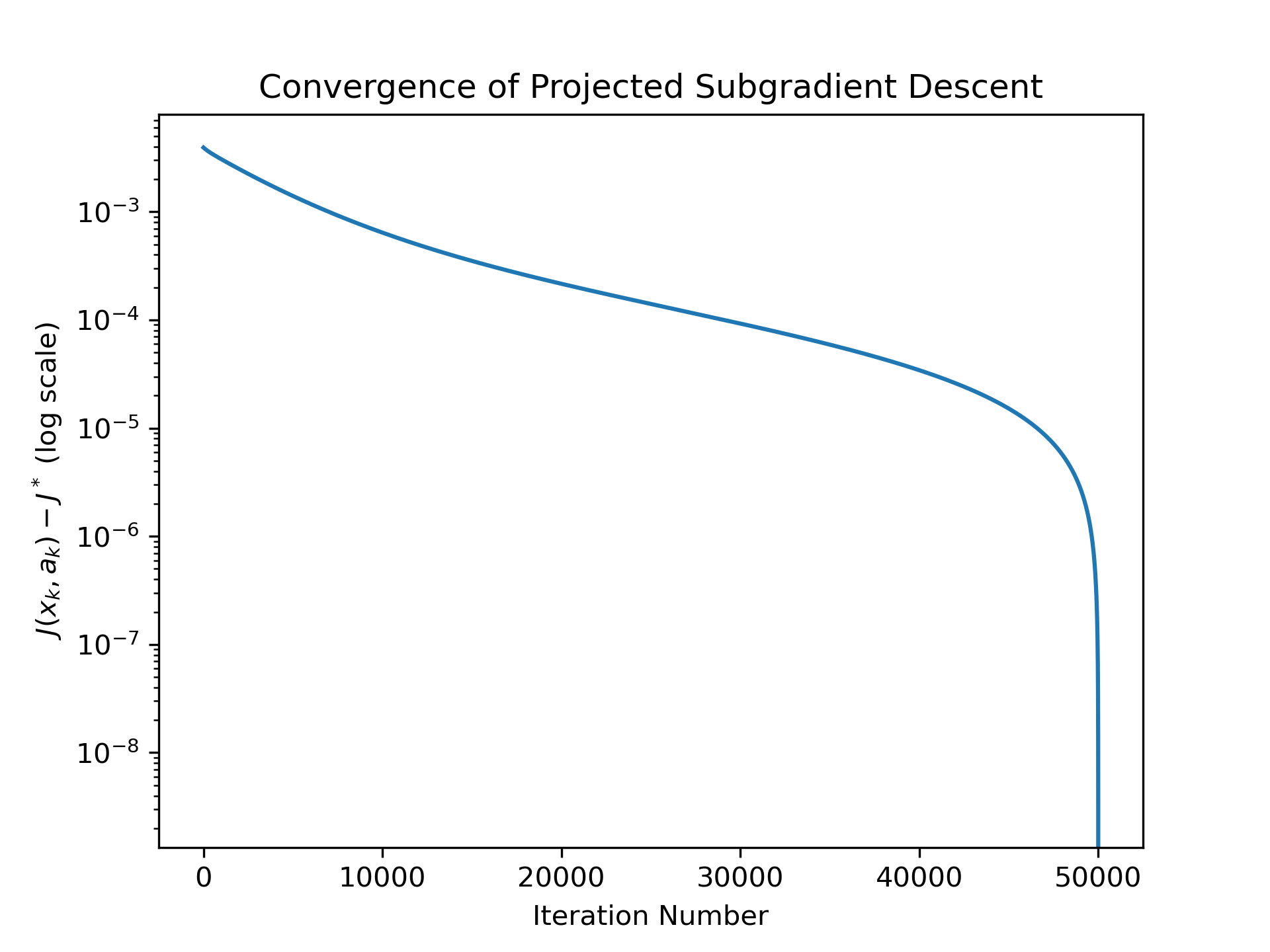} %
        \caption{Convergence of the projected subgradient descent algorithm with a fixed step-size of $\eta = 0.001$, and $c=0.1,\varepsilon=0.03,q_0=0.024$. The y-axis plots the suboptimality gap, $J(x_k, a_k) - J^*$, on a logarithmic scale against the iteration number $k$. Although $h$ is not strongly concave, the algorithm still converges at a super-exponential rate. } %
        \label{fig:convergence1} %
    \end{figure}

We evaluate the out-of-sample performance of our DRO estimator against the standard Sample Average Approximation (SAA) baseline. To create an accurate testing environment, we sample a large dataset of $3 \times 10^6$ returns from the true distribution to serve as the out-of-sample distribution. We run $N=100$ simulations.  In each simulation, a new training set of 1,000 samples is drawn to fit both the DRO and SAA estimators, which are then evaluated on the large out-of-sample test set. A 100 repetition allows us to get the mean and 20th-80th percentiles of the out-of-sample disutility. The left side of Figure \ref{fig:combined-models} plots these results, illustrating how the out-of-sample performance (left column) and the corresponding average portfolio allocation (right column) evolve with the radius parameter $c$ for different half interval width $\varepsilon$. The plots clearly show that the DRO estimator's performance is sensitive to these parameters, often outperforming the SAA (the blue line lies below the red dashed line) up to an optimal $c$. The green line represents the disutility of the best portfolio under the true distribution.

\subsection{Mean-CVaR Disutility}

We use the same setup as in the Mean-Var problem and set $\rho=10$ and $p=0.95$ to indicate high risk aversion, as in the Mean-Var problem. The right side of Figure \ref{fig:combined-models} again shows that the DRO estimator's performance often outperforms the SAA up to an optimal level of $c$.

\begin{figure}[htbp]
    \centering %
    \includegraphics[width=\textwidth]{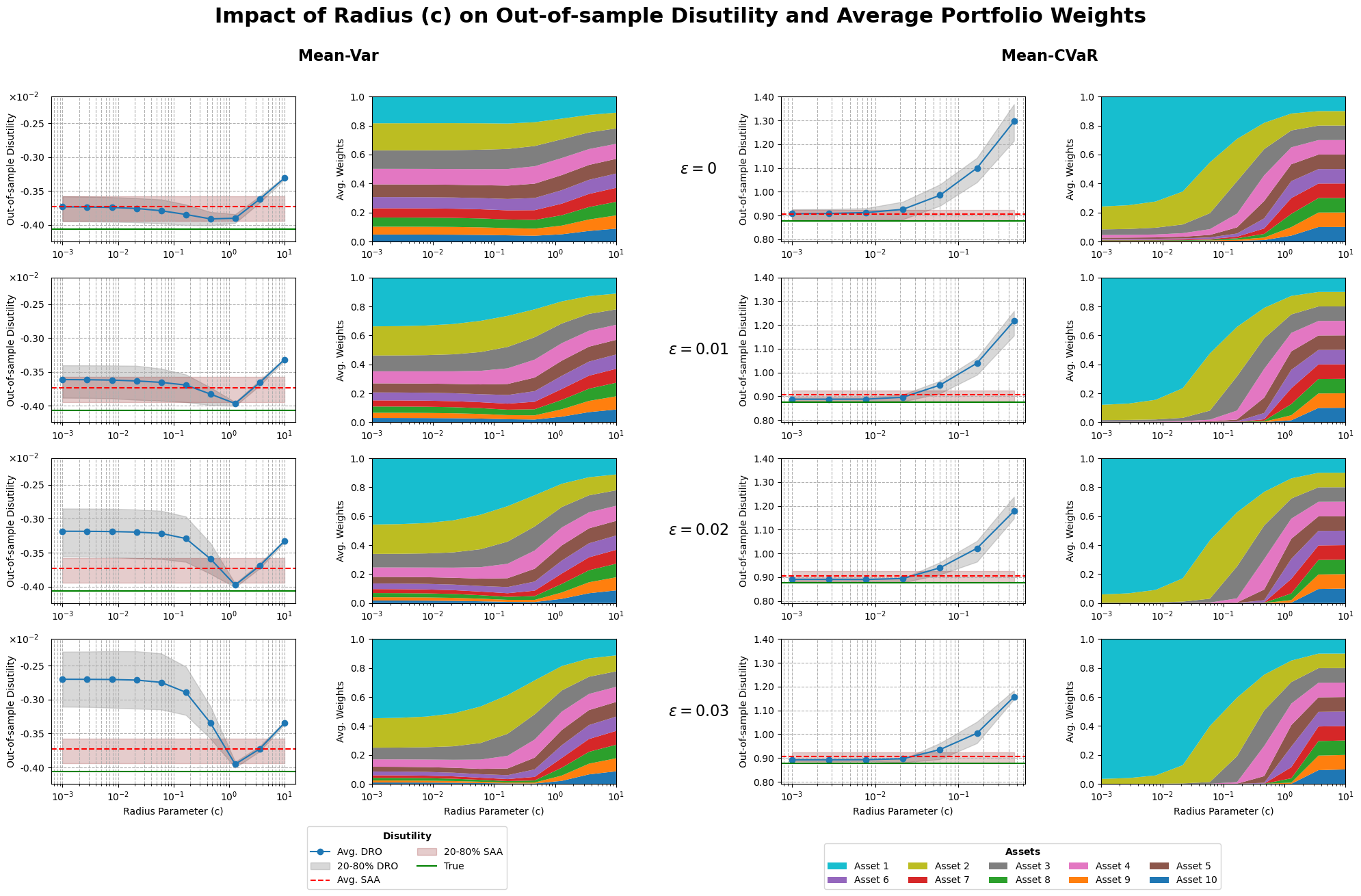}
    \caption{Comparison of Out-of-sample performance and portfolio allocations for (a) the Mean-Var model and (b) the Mean-CVaR model. Each row corresponds to a different half interval width $\varepsilon$, ranging in the set $\{0,0.01,0.02,0.03\}$. The first and third columns correspond to the Mean DRO Out-of-sample disutility (blue line) and its 20th-80th percentile range (shaded area) across 100 simulations (each using 1000 in-sample points). The mean SAA Out-of-sample disutility is the dashed red line, with its 20th-80th percentile range shaded in red. The second and fourth columns correspond to the corresponding average portfolio allocation. The green line represents the disutility of the best portfolio under the true distribution.}
    \label{fig:combined-models} %
\end{figure}

\section{Conclusion}
In this article, we obtained min-max formulations for the Mean-Var and Mean-CVaR DRO problems \eqref{eq:main_problem} \eqref{mean-cvar} with a mixture structure, which can then be solved by projected subgradient descent. Our algorithm guarantees an $O(\frac{1}{\sqrt{T}})$ convergence rate, while for the Mean-Var problem, it can attain exponential convergence rates. In Section 4, we presented numerical examples demonstrating the attainment of exponential rates even when some assumptions are violated. Additionally, we demonstrated the efficacy of our DRO method by comparing out-of-sample disutility with that of the SAA method. In both cases, we show, both theoretically and empirically, that as the radius of our ball tends to infinity, the DRO portfolio converges to the 1/N portfolio.

We briefly discuss the case when a short sale is permitted. Normally, there is a limit posed on the leverage position, i.e., we constrain our portfolio to be $\mathbf{1}'\x=1; \x \geq v$, with $v$ a negative number.  In this case, all the results, including the min-max reformulations and the convergence rates, still hold, since $\x$ lies in a compact set and we can still restrict $a,\tau$ to a compact set. In the idealized scenario when we allow short sale without a limit, i.e., the only constraint we put on $\x$ is $\mathbf{1}'\x=1$, we cannot restrict $a$ and $\tau$ to a compact set. A key step for obtaining similar min-max reformulations as in \eqref{finalopt} and \eqref{minimax_cvar}, only with changes to the domain of the minimum problem, is by showing the ambiguity set $\calS$ is compact with respect to the corresponding Wasserstein metric topology, and applying Sion's minimax theorem (Theorem \ref{sion}). In the Mean-Var problem, the final reformulation becomes 
\begin{equation}
\label{NoRestrictionVar}
\inf_{\mathbf{1}'\x=1,a \in \Real}\;\max_{q \in [q_0-\varepsilon,\, q_0+\varepsilon]} \bigg[ (1-q)\cdot\EP{P_{N}}{\psi_{a,\gamma}(\mathbf{x}'\mathbf{R})} + q\cdot V(q, \mathbf{x},a) \bigg],
\end{equation}
where 
\[
\psi_{a,\gamma}(y) := (y-a)^2 - \gamma y,
\]
and
\[
V(q, \mathbf{x},a)=\left( r(q)\|\mathbf{x}\| + \sqrt{\x'\Sigma_{P_{S_0}} \x + \left(\x' \mu_{P_{S_0}} - \frac{2a+\gamma}{2}\right)^2} \right)^2 - a\gamma - \frac{\gamma^2}{4},
\]
In the Mean-CVaR problem, the final min-max problem becomes
\begin{equation}
\label{NoRestrictionCVaR}
    \inf_{\mathbf{1}'\x=1,\tau \in \Real}\left\{ \rho\tau + \sup_{q \in [q_0\pm\epsilon]} \left( (1-q)\EP{P_{N}}{\ell_{\x,\tau}(\mathbf{R})} + q\left[\EP{P_{S_0}}{\ell_{\x,\tau}(\mathbf{R})} + r(q)(1+\frac{\rho}{1-p})\|\mathbf{x}\|_\infty\right] \right) \right\}.    
\end{equation}

In \eqref{NoRestrictionVar}, all our results still hold since we can restrict $\x$ and $a$ to a compact set without affecting the infimum. However, in \eqref{NoRestrictionCVaR}, it is possible that the minimum is $-\infty$, attained at an arbitrarily large leveraged position. %

\section{Appendix}

\subsection{Some theorems used in Sections 2 and 3}

\begin{Theorem}[Theorem~1 in \citet{gao2023distributionally}]
\label{duality}
For any Polish space $(\Xi, d)$ and measurable function $\Psi$ with $p \geq 1$,
    \[
\sup_{\mu \in \mathcal{P}(\Xi)} 
\left\{ \mathbb{E}_{\mu}[\Psi(x,\xi)] : W_p(\mu,\nu) \le \theta \right\}
= 
\min_{\lambda \ge 0} 
\left\{ 
\lambda \theta^p 
- \int_{\Xi} 
\inf_{\xi \in \Xi} 
\left[ 
\lambda d^p(\xi,\zeta) - \Psi(x,\xi)
\right] 
\nu(d\zeta)
\right\}.
\]\qed
\end{Theorem}

\begin{Theorem}[{Proposition A.3.2 in \citet{bertsekas2009convex}: Danskin’s Theorem}]
\label{danskin}
Let $Z \subset \mathbb{R}^m$ be a compact set, and let $\phi : \mathbb{R}^n \times Z \mapsto \mathbb{R}$ be continuous and such that $\phi(\cdot, z) : \mathbb{R}^n \mapsto \mathbb{R}$ is convex for each $z \in Z$.

\begin{enumerate}[label=(\alph*)]
    \item The function $f : \mathbb{R}^n \mapsto \mathbb{R}$ given by
\begin{equation}
\label{maxDanskin}
f(x) = \max_{z \in Z} \phi(x, z)
\end{equation}
is convex and has directional derivative given by
\begin{equation*}
f'(x; y) = \max_{z \in Z(x)} \phi'(x, z; y),
\end{equation*}
where $\phi'(x, z; y)$ is the directional derivative of the function $\phi(\cdot, z)$ at $x$ in the direction $y$, and $Z(x)$ is the set of maximizing points in \eqref{maxDanskin}:
\[
Z(x) = \left\{ z \,\middle|\, \phi(x, z) = \max_{z \in Z} \phi(x, z) \right\}.
\]

In particular, if $Z(x)$ consists of a unique point $\bar{z}$ and $\phi(\cdot, \bar{z})$ is differentiable at $x$, then $f$ is differentiable at $x$, and
\[
\nabla f(x) = \nabla_x \phi(x, \bar{z}),
\]
where $\nabla_x \phi(x, \bar{z})$ is the vector with coordinates
\[
\frac{\partial \phi(x, \bar{z})}{\partial x_i}, \quad i = 1, \ldots, n.
\]

    \item If $\phi(\cdot, z)$ is differentiable for all $z \in Z$ and $\nabla_x \phi(x, \cdot)$ is continuous on $Z$ for each $x$, then
\begin{equation*}
\partial f(x) = \mathrm{conv}\bigl\{ \nabla_x \phi(x, z) \,\big|\, z \in Z(x) \bigr\}, \qquad \forall x \in \mathbb{R}^n.
\end{equation*}
\end{enumerate}
\end{Theorem}

\begin{lemma}[Berge's {Maximum Theorem}, \citet{Berge}]
    \label{Berge}
Let $X$ and $\Theta$ be topological spaces, 
$f : X \times \Theta \to \mathbb{R}$ be a continuous function on the product $X \times \Theta$, 
and $C : \Theta \rightrightarrows X$ be a compact-valued correspondence such that 
$C(\theta) \neq \emptyset$ for all $\theta \in \Theta$. 

Define the \emph{marginal function} (or \emph{value function}) 
$f^* : \Theta \to \mathbb{R}$ by
\[
f^*(\theta) = \sup \{ f(x,\theta) : x \in C(\theta) \}
\]
and the \emph{set of maximizers} 
$C^* : \Theta \rightrightarrows X$ by
\[
C^*(\theta) = \operatorname*{arg\,max}_{x \in C(\theta)} f(x,\theta)
= \{ x \in C(\theta) : f(x,\theta) = f^*(\theta) \}.
\]

If $C$ is continuous (i.e.\ both upper and lower hemicontinuous) at $\theta$, 
then the value function $f^*$ is continuous, 
and the set of maximizers $C^*$ is upper hemicontinuous with nonempty and compact values. 
As a consequence, the $\sup$ may be replaced by $\max$. \qed
\end{lemma}

\subsection{Convergence rate of Projected Subgradient Descent}

The fundamental complexity result for the projected subgradient descent algorithm is the following:

\begin{Theorem}[Theorem 3.2 in \citet{bubeck2015convex}]
\label{slowConverge}
Let J be a convex function on $\mathcal{X}$, a compact and convex Euclidean set. Let $x_k$ be the vector of decision variables at iteration $k$, and let $x^*$ be an optimal solution. Assume J is $G$-Lipschitz in $\mathcal{X}$ and $\|x_0-x^*\| \leq R$ . Projected gradient descent with constant step size $\eta_k=\frac{R}{G \sqrt{T}}$ (i.e. $x_{k+1} = \text{Proj}_\mathcal{X}(x_k - \eta_k \mathbf{g}_k)$) satisfies
\[J\left(\frac{1}{T}\sum_{i=1}^{T}x_i\right)-J(x^*) \leq \frac{RG}{\sqrt{T}}.\]
\qed
\end{Theorem}

However, the above theorem has a relatively slow convergence rate of $O(1/\sqrt{T})$. %
To strengthen the conclusions, we need stronger assumptions on $J$, and for which we introduce the following definitions:
\begin{definition}[Strong Convexity] \label{sconv}
    For $\alpha >0$, a function $f: \Omega \subset \Real^n \rightarrow \Real$ is $\alpha-$strongly convex if $f-\frac{\alpha}{2} \Vert \x\Vert^2$ is convex. 
\end{definition}

\begin{definition}[$\beta$-smooth] \label{smooth}
We say that a continuously differentiable convex function $f$ is $\beta$-smooth 
if the gradient $\nabla f$ is $\beta$-Lipschitz, that is
\[
\|\nabla f(x) - \nabla f(y)\| \le \beta \|x - y\|.
\]
\end{definition}
Note that if $f$ is twice differentiable, then strong convexity is equivalent to the eigenvalues of the Hessians being at least $\alpha$. On the other hand, $\beta$-smoothness is equivalent to the eigenvalues of the Hessians being at most $\beta$. The ratio $\beta/\alpha$ is known as the condition number of $f$.

\begin{Theorem}[Theorem 3.9 in \citet{bubeck2015convex}]
\label{mediumConverge}
Let $J$ be $\alpha$-strongly convex and $L$-Lipschitz on $\mathcal{X}$, a compact and convex Euclidean set. 
Then projected subgradient descent with time-varying step size
\[
\eta_s = \frac{2}{\alpha (s+1)}
\]
satisfies
\[
J\left( \sum_{s=1}^T \frac{2s}{T(T+1)} x_s \right) - J(x^*)
\;\le\;
\frac{2L^2}{\alpha (T+1)} .
\]

\qed
\end{Theorem}

\begin{Theorem}[Theorem 3.10 in \citet{bubeck2015convex}]
\label{fastConverge}
Let $J$ be $\alpha$-strongly convex and $\beta$-smooth on $\mathcal{X}$, a compact and convex Euclidean set, with $\kappa$ denoting its condition number.  Then projected gradient descent with a constant step size of $(1/\beta)$ satisfies for $T \ge 0$,
\[
\|\theta_{T} - \theta^*\|^2 \le \exp\!\left(-\frac{T}{\kappa}\right) \|\theta_0 - \theta^*\|^2.
\]
Note that 
\[
J(\theta_{T}) - J(\theta^*) \le \frac{\beta}{2}\,
\|\theta_{T} - \theta^*\|^2  \leq \frac{\beta}{2}\, \exp\!\left(-\frac{T}{\kappa}\right) \|\theta_0 - \theta^*\|^2.
\]
\qed
\end{Theorem}

\subsection{Minimax Theorems}

\begin{Theorem}[{Theorem 3} in \citet{Terk1972Minimax}]
    Let $X$ be a compact connected topological space, let $Y$ be a set, and let $f: X \times Y \to \mathbb{R}$ be a function satisfying:
\begin{enumerate}
    \item[(i)] For any $y_1, y_2 \in Y$ there exists $y_0 \in Y$ such that
    $$ f(x, y_0) \ge \left(\frac{1}{2}\right)(f(x,y_1)+f(x,y_2)) $$
    for all $x \in X$.
    \item[(ii)] Every finite intersection of sets of the form $\{x \in X : f(x,y) \le \alpha\}$, with $(y,\alpha) \in Y \times \mathbb{R}$, is closed and connected.
\end{enumerate}
Then
$$ \sup_{y \in Y} \min_{x \in X} f(x,y) = \min_{x \in X} \sup_{y \in Y} f(x,y). $$
\end{Theorem}

\begin{Theorem}[Corollary 3.3 of Theorem 3.4 in \citet{sion1958general}]
\label{sion}
Let $X$ and $Y$ be convex subsets of topological vector spaces, with $X$ compact, 
and let $f : X \times Y \to \mathbb{R}$ satisfy:
\begin{itemize}
    \item[(i)] For each $x \in X$ the function $y \mapsto f(x,y)$ is upper semi-continuous and quasi-concave on $Y$;
    \item[(ii)] For each $y \in Y$ the function $x \mapsto f(x,y)$ is lower semi-continuous and quasi-convex on $X$.
\end{itemize}
Then
\[
    \sup_{y} \min_{x} f = \min_{x} \sup_{y} f.
\]
\end{Theorem}

\bibliographystyle{plainnat}
\bibliography{ref}

\end{document}